\documentclass[authoryear, review]{elsarticle}
\usepackage[utf8]{inputenc}

\usepackage{dsfont}
\usepackage{latexsym}
\usepackage{amssymb}
\usepackage{amsmath}
\usepackage{amsthm}
\usepackage{algorithmic}
\usepackage{paralist}
\usepackage{algorithm}
\usepackage{upgreek}
\usepackage{array}
\newcommand{\PreserveBackslash}[1]{\let\temp=\\#1\let\\=\temp}
\newcolumntype{C}[1]{>{\PreserveBackslash\centering}p{#1}}
\usepackage{graphicx}
\usepackage{multirow}

\usepackage{anysize}
\usepackage{fancyhdr}
\usepackage{color}
\usepackage{tocbibind}

\usepackage{fancyhdr}
\usepackage{epsfig}

\usepackage{xcolor}
\usepackage{lineno,hyperref}
\hypersetup{colorlinks = true, linkcolor = blue, anchorcolor =red, citecolor = blue,filecolor = red, urlcolor = red, pdfauthor=author}
\usepackage{tikz}
\usepackage{rotating}
\usepackage{url}
\usepackage{ulem}
\usepackage{footnote}

\newtheorem{thm}{Theorem}

\newtheorem{prop}[thm]{Proposition}
\newtheorem*{lemma*}{Lemma}

\numberwithin{equation}{section}
\usepackage{float}

\begin{document}

\begin{frontmatter}
\title{Benders decomposition for Network Design Covering Problems}

\author[ad0,ad00]{V\'ictor Bucarey}
\ead{vbucarey@vub.be, victor.bucarey@uoh.cl}
\author[ad2,ad3]{Bernard Fortz}
\ead{bernard.fortz@ulb.ac.be}
\author[ad1]{Natividad González-Blanco \corref{cor1}
}
\ead{ngonzalez2@us.es}
\author[ad2,ad3]{Martine Labb\'e}
\ead{mlabbe@ulb.ac.be}
\author[ad1]{Juan A. Mesa}
\ead{jmesa@us.es}

\address[ad0]{Data Analytics Laboratory, Vrije Universiteit Brussel, Brussels, Belgium.} 
\address[ad00]{Institute of Engineering Sciences, Universidad de O'Higgins, Chile.} 
\address[ad1]{Departamento de Matemática Aplicada II de la Universidad de Sevilla, Sevilla, Spain.}
\address[ad2]{D\'epartement d'Informatique, Universit\'e Libre de Bruxelles, Brussels, Belgium.}
\address[ad3]{Inria Lille-Nord Europe, Villeneuve d'Ascq, France.}
\cortext[cor1]{Corresponding author}

\begin{abstract}
We consider two covering variants of the network design problem. We are given a set of origin/destina\-tion pairs, called O/D pairs, and each such O/D pair is covered if there exists a path in the network from the origin to the destination whose length is not larger than a given threshold. In the first problem, called the Maximal Covering Network Design problem, one must determine a network that maximizes the total fulfilled demand of the covered O/D pairs subject to a budget constraint on the design costs of the network. In the second problem, called the Partial Covering Network Design problem, the design cost is minimized while a lower bound is set on the total demand covered.\\
After presenting formulations, we develop a Benders decomposition approach to solve the problems. Further, we consider several stabilization methods to determine Benders cuts as well as the addition of cut-set inequalities to the master problem. We also consider the impact of adding an initial solution to our methods. Computational experiments show the efficiency of these different aspects.
\end{abstract}

\begin{keyword}
Facility planning and design \sep
Benders decomposition \sep
Network design \sep
Rapid transit network \end
{keyword}

\end{frontmatter}

\section{Introduction}
Network design is a broad and spread subject whose models often depend on the field in which they are applied. A classification of the basic problems of network design was done by \cite{MagnantiWong1984} where some classical graph problems as the minimal spanning tree, Steiner tree, shortest path, facility location, and traveling salesman problems are included as particular cases of a general mathematical programming model. Since the construction of a network often costs large amount of money and time, decisions on network design are a  crucial step when planning  networks. Thus, network design is applied in a wide range of fields: transportation, telecommunication, energy, supply chain, geostatistics, evacuation, monitoring, etc.
Infrastructure network design constitutes a major step in the planning of a transportation network since the performance, the efficiency, the robustness and other features strongly depend on the selected nodes and the way of connecting them, see \cite{guihaire2008transit}. For instance, the main purpose of a rapid transit network is to improve the mobility of the inhabitants of a city or a metropolitan area. This improvement could lead to lower journey times, less pollution and/or less energy consumption which drives the communities to a more sustainable mobility.

Since it is generally too expensive to connect all the potential nodes, one must determine a subnetwork that serves at best the traffic demand. Depending on the application, different optimality measures can be considered. In particular, in the field of transportation, and especially in the area of passenger transportation, the aim is to get the infrastructure close to potential customers. In this framework, \cite{schmidt2014location} propose to minimize the maximum routing cost for an origin-destination pair when using the new network. Alternatively, the traffic between an origin and a destination may be considered as captured if the cost or travel time when using the network is not larger than the cost or travel time of the best alternative solution (not using the new network). In this case, \cite{perea2020transportation} and \cite{garcia2013grasp} propose to select a sub(network) from an underlying network with the aim of capturing or covering as much traffic as possible for a reasonable construction cost. This paper is devoted to this problem, called the \textbf{M}aximum \textbf{C}overing Network Design Problem $(MC)$ as well as to the closely related problem called, \textbf{P}artial \textbf{C}overing Network Design Problem $(PC)$. The latter aims to minimize the network design cost for constructing the network under the constraint that a minimum percentage of the total traffic demand is covered.

Covering problems in graphs have attracted the attention of researchers since the middle of the last century. As far as the
authors are aware the first papers on the vertex-covering problem were due to \cite{berge1957two} and \cite{norman1959algorithm} in the late 50s. This problem is related to the set-covering problem in which a family of sets is given and the minimal subfamily whose union contains all the elements is sought for. In \cite{hakimi1965optimum} the vertex-covering problem was formulated as an integer linear programming model and solved by
using Boolean functions. \cite{toregas1971location} applied the vertex-covering problem to the location of emergency services. They assumed that a vertex is covered if it is within a given coverage distance. \cite{Church1974} introduced the maximal covering location problem by fixing the number of facilities to be located. Each vertex has an associated population and the objective is to cover the maximum population within a fixed distance threshold. Since then many variants and extensions of the vertex-covering and maximal covering  problems have been studied (see \cite{GarciaMarin2020}.)

When designing an infrastructure network, the demand is given by pairs of origin-destination points, called O/D pairs, and each such pair has an associated weight representing the traffic between the origin and the destination. Usually, this demand is encoded using an origin-destination matrix. When planning a new network, often there exists a network already functioning and offering its service to the same set of origin-destination pairs. For example, a new rapid transit system may be planned in order to improve the mobility of a big city or metropolitan area, in which there already exists another transit system, in addition to the private transportation system. This current transit system could be more dense than the planned one but slower since it uses the same right-of-way as the private traffic system. Thus, in some way both systems compete with each other and both compete with the private system of transportation. A similar effect occurs with mobile telecommunication operators. Therefore, the traffic between an origin and a destination is distributed among the several systems that provide the service.\\
There are mainly two ways of allocating the share of each system. The first one is the binary all-or-nothing way where the demand is only covered by one of the proposed modes. Typically, the demand is covered if the demand points are served within a range of quality service, as in \cite{church1974maximal}. The second one is based on some continuous function, using, for example, a multi-logit probability distribution, as in \cite{cascetta2009transportation}. In this case, the demand is shared between the different systems. Both allocation schemes are based on the comparison of distances, times, costs, generalized costs or utilities. In this paper, we consider a binary one, where each O/D pair is covered only if the time spent to travel from its origin to its destination in the network is below a threshold. This threshold represents the comparison between the time spent in the proposed network and a private mode, assigning the full share to the most beneficial one.

Since most network design problems are NP-hard (see e.g. \cite{perea2020transportation}), recent research efforts have been oriented to apply metaheuristic algorithms to obtain good solutions in a reasonable computational time. Thus, in the field of transportation network design, Genetic Algorithms \citep{krol2019design}, Greedy Randomized Adaptive Search Procedures \citep{garcia2013grasp}, Adaptive Large Neighborhood Search Procedures \citep{canca2017adaptive} and Matheuristics \citep{canca2019integrated} have been used to solve rapid transit network design problems and applied to medium-sized instances.

In this paper, after presenting models for problems $(MC)$ and $(PC)$, we propose exact methods based on Benders decomposition 
\citep{benders1962partitioning}. This type of decomposition has been applied to many problems in different fields, see \cite{rahmaniani2017benders} for a recent literature review on the use of Benders decomposition in combinatorial optimization. One recent contribution applied to set covering and maximal covering location problems appear in \cite{Cordeau2019benders}. The authors propose different types of normalized Benders cuts for these two covering problems.

Benders decomposition for network design problems has been studied since the 80s. In \cite{magnanti1986tailoring}, the authors minimize the total construction cost of an uncapacitated network subject to the constraint that all O/D pairs must be covered. Given the structure of the problem, the Benders reformulation is stated with one subproblem for each O/D pair. A Benders decomposition for a multi-layer network design problem is presented in \cite{FORTZ2009359}. Benders decomposition was also applied in \cite{botton2013benders} in the context of designing survivable networks. In \cite{Costa2009benders}, a multi-commodity capacitated network design problem is studied and the strength of different Benders cuts is analysed. In \cite{marin2009urban} a multi-objective approach is solved through Benders decomposition. The coverage is maximized and the total cost design is minimized. To the best of our knowledge, we apply for the first time a branch-and-Benders-cut approach to network design coverage problems. We also give a detailed study of some normalization techniques for Benders cuts in this context, including {\it facet-defining cuts} \citep{Conforti2019facet}. These cuts are a stronger version of the cuts proposed by \cite{ben2007acceleration}.

This paper presents several contributions. First, we present new mathematical integer formulations for the network design problems $(MC)$ and $(PC)$. The formulation for $(MC)$ is stronger than a previously proposed one, see e.g. \cite{marin2009urban} and \cite{garcia2013grasp} (although the proposed formulation was not the main purpose of the latter), while $(PC)$ was never studied to the best of our knowledge. Our second contribution consists of polyhedral properties that are useful from the algorithmic point of view. A third contribution is the study of exact algorithms for the network design based on different Benders implementations. We propose a normalization technique and we consider the facet-defining cuts. Our computational experiments show that our Benders implementations are competitive with exact and non-exact methods existing in the literature and even comparing with the exact method of Benders decomposition existing in \texttt{CPLEX}.

The structure of the paper is as follows. In Section \ref{s:models}, we present mixed integer linear formulations for $(MC)$ and $(PC)$. We also study some polyhedral properties of the formulations and propose a simple algorithm to find an initial feasible solution for both problems. In Section \ref{sec:BImp}, we study different Benders implementations and some algorithmic enhancements. Also, we discuss some improvements based on cut-set inequalities. A computational study is detailed in Section \ref{s:computational}. Finally, our conclusions are presented in Section \ref{s:conclusion}.

\section{Problem formulations and some properties}\label{s:models}
In this section we present mixed integer linear formulations for the Maximal Covering Network Design Problem $(MC)$ and the Partial Set Covering Network Design Problem $(PC)$. We also describe some pre-processing methods and finish with some polyhedral properties. We first introduce some notation.

We consider an undirected graph denoted by $\mathcal{N} = (N,E)$, where $N$ and $E$ are the sets of potential nodes and edges that can be constructed. Each element $e \in E$ is denoted by $\{i,j\}$, with $i,j \in N$. We use the notation $i\in e$ if node $i$ is a terminal node of $e$. Let $\delta(i)$ be the set of edges incident to node $i$.

The mobility patterns are represented by a set $W\subset N\times N$ of O/D pairs. Each $w = (w^s,w^t)\in W$ is defined by an origin node $w^s \in N$, a destination node $w^t\in N$, an associated demand $g^w > 0$ and a utility $u^w > 0$. This utility translates the fact that there already exists a different network competing with the network to be constructed in an all-or-nothing way. In other words, an O/D pair $(w^s,w^t)$ will travel on the newly constructed network if it contains a path between $w^s$ and $w^t$ of length shorter than or equal to the utility $u^w$. We then say that the O/D pair is covered. In terms of the transportation area, the existing network represents a private transportation mode, the planned one represents a public transportation mode and the parameter $u^w,\,w\in W$, refers to the utility of taking the private mode.

Costs for building nodes, $i\in N$, and edges, $e \in E$, are denoted by $b_i$ and $c_e$, respectively. The total construction cost cannot exceed the budget $C_{max}$. For example, in the context of constructing a transit network, each node cost may represent the total cost of building one station in a specific location in the network. On the other hand, each edge cost represents the total cost of linking two stations. 

For each $e=\{i,j\}\in E$, we define two arcs: $a=(i,j)$ and $\hat{a} =(j,i)$. The resulting set of arcs is denoted by $A$. The length of arc $a \in A$ is denoted by $d_a$. For each O/D pair $w \in W$ we define a subgraph $\mathcal{N}^w = (N^w, E^w)$ containing all feasible nodes and edges for $w$, i.e. that belong to a path in $\mathcal{N}$ whose total length is lower than or equal to $u^w$. We also denote $A^w$ as the set of feasible arcs. In Section \ref{s:pre}, we describe how to construct these subgraphs. We use notation $\delta_w^+(i)$ ($\delta_w^-(i)$ respectively) to denote the set of arcs going out (in respectively) of node $i \in N^w$. In particular, $\delta_w^-(w^s) = \emptyset$ and $\delta_w^+(w^t) = \emptyset$. We also denote by $\delta_w(i)$ the set of edges incident to node $i$ in graph $\mathcal N^w$.

\subsection{Mixed Integer Linear Formulations}\label{s:formulations}

We first present a formulation of the \textit{Maximal Covering Network Design Problem} $(MC)$, whose aim is to design an infrastructure network maximizing the total demand covered subject to a budget constraint: 

\begin{align}
(MC) \quad \max_{\boldsymbol{x},\boldsymbol{y},\boldsymbol{z},\boldsymbol{f}} &\quad \sum_{w \in W} g^w z^w &\label{eq:MCobj}\\ 
\mbox{s.t. } &\quad \sum_{e \in E} c_e x_e + \sum_{i \in N} b_i y_i \le C_{max}, \label{eq:budget} \\
&\quad x_e \le y_i,\hspace{8cm} e \in E, i \in e, \label{rel_xy}\\
&\quad \sum_{a \in \delta_w^+(i)} f^w_a -\sum_{a \in \delta_w^-(i)} f^w_a = 
 \begin{cases}
z^w, &\text{if $i = w^s$},\\
-z^w, &\text{if $i = w^t$}, \\
0, & \text{otherwise}, 
\end{cases} \hspace{1.8cm} w \in W, i \in N^w, \label{flow_cons}\\
&\quad f^w_{a}+f^w_{\hat{a}} \le x_{e}, \hspace{2cm} w \in W,\, e=\{i,j\} \in E^w: a=(i,j), \, \hat{a}=(j,i),\label{Capacity} \\
&\quad \sum_{a \in A^w} d_a f^w_a \le u^w z^w, \hspace{7cm} w \in W, \label{Utility} \\
&\quad y_i,\, x_e,\, z^w \in \{0,1\}, \hspace{5cm} i\in N,\, e\in E,\, w\in W,\label{eq:nat}\\ 
&\quad f_a^w \in \{0,1\}, \hspace{7cm} a \in A^w, w\in W,\label{eq:nat-f}
\end{align}
\noindent where $y_i$ and $x_e$ represent the binary design decisions of building node $i$ and edge $e$, respectively. Mode choice variable $z^w$ takes value $1$ if the O/D pair $w$ is covered and $0$ otherwise. Variables $f_a^w$ are used to model a path between $w^s$ and $w^t$, if possible. Variable $f_a^w$ takes value $1$ if arc $a$ belongs to the path from $w^s$ to $w^t$, and $0$ otherwise. Each variable $f^w_a$, such that $a\notin A^w$, is set to zero.

The objective function (\ref{eq:MCobj}) to be maximized represents the demand covered. Constraint (\ref{eq:budget}) limits the total construction cost. Constraint (\ref{rel_xy}) ensures that if an edge is constructed, then its terminal nodes are constructed as well. For each pair $w$, expressions (\ref{flow_cons}), (\ref{Capacity}) and (\ref{Utility}) guarantee demand conservation and link flow variables $f_a^w$ with decision variables $z^w$ and design variables $x_e$. Constraints (\ref{Capacity}) are named capacity constraints and they force each edge to be used only in one direction at most. Constraints (\ref{Utility}) referenced as mode choice constraints, put an upper bound on the length of the path for each pair $w$. This ensures variable $z^w$ to take value $1$ only if there exists a path between $w^s$ and $w^t$ with length at most $u^w$. This path is represented by variables $f_a^w$. 
Remark that several paths with length not larger than $u^w$ may exists for a given design solution $x,y$. Then, the values of the  flow variables $f_a^w$ will describe one of them and the path choice has no influence on the objective function value \eqref{eq:MCobj}. 
Finally, constraints (\ref{eq:nat}) and (\ref{eq:nat-f}) state that variables are binary.

The \textit{Partial Covering Network Design Problem} $(PC)$, which minimizes the total construction cost of the network subject to a minimum coverage level of the total demand, can be formulated as follows:
\begin{align}
(PC) \quad \min_{\boldsymbol{x},\boldsymbol{y},\boldsymbol{z},\boldsymbol{f}} &\quad \sum_{i \in N} b_i y_i + \sum_{e \in E} c_e x_e &\label{eq:PCobj}\\ 
\mbox{s.t. }& \quad \sum_{w \in W} g^w z^w \ge \beta \, G, \label{eq:PCconst} \\
 & \quad \mbox{Constraints (\ref{rel_xy}), (\ref{flow_cons}), (\ref{Capacity}), (\ref{Utility}), (\ref{eq:nat}), (\ref{eq:nat-f}), }& \nonumber
\end{align}
\noindent where $\beta\in (0,1]$ and $G=\sum\limits_{w\in W}g^w$. Here, the objective function (\ref{eq:PCobj}) to be minimized represents the design cost. Constraint (\ref{eq:PCconst}) imposes that a proportion $\beta$ of the total demand is covered.

In the previous works by \cite{marin2009urban} and \cite{garcia2013grasp}, constraints \eqref{Capacity} and \eqref{Utility} are formulated in a different way. For example, in \cite{garcia2013grasp}, these constraints were written as 
\begin{align}
 & f_a^w + z^w - 1 \le x_a,& \kern-2cm w \in W,\, e=\{i,j\} \in E^w: a=(i,j), \label{eq:cap_ardila} \\
 & \sum_{a\in A^w} d_a f_a^w + M(z^w -1)\le u^w z^w, & w \in W,\label{eq:uti_ardila}
\end{align}
\noindent where the design variable $x_a$ is defined for each arc. Given that $z^w-1 \le 0$, expressions \eqref{Capacity} and \eqref{Utility} are stronger than \eqref{eq:cap_ardila} and \eqref{eq:uti_ardila}, respectively. 

In addition, constraint \eqref{eq:uti_ardila} involves a \textit{``big-M''} constant. Our proposed formulation does not need it, which avoids the numerical instability generated by this constant. As we will see in Section 4.2, we observed that our proposed formulation is not only stronger than the one proposed in \cite{garcia2013grasp}, but it is also computationally more efficient. In consequence, we only focus our analysis on our proposed formulation.

Another observation is that constraints \eqref{Capacity} are a reinforcement of the usual capacity constraints $f_a^w \leq x_e$ and $f_{\hat{a}}^w \leq x_e$. In most applications where flow or design variables appear in the objective functions, the disaggregated version is sufficient to obtain a valid model as subtours are naturally non-optimal. However, it is not the case in our model, and there exist optimal solutions with subtours if the disaggregated version of \eqref{Capacity} is used. Such a strengthening was already introduced in the context of uncapacitated network design, see e.g. \cite{balakrishnan1989dual}, and Steiner trees, see e.g. \cite{sinnl2016node} and \cite{fortz2020steiner}.

\subsection{Pre-processing methods}\label{s:pre}

In this section we describe some methods to reduce the size of the instances before solving them. First, we describe how to build each subgraph $\mathcal{N}^w=(N^w,E^w)$. Then for each problem, $(MC)$ and $(PC)$, we sketch a method to eliminate O/D pairs which will never be covered.

To create $\mathcal N^w$ we only consider useful nodes and edges from $\mathcal N$. For each O/D pair $w$, we eliminate all the nodes $i\in N$ that do not belong to any path from $w^s$ to $w^t$ shorter than $u^w$. Then, we define $E^w$ as the set of edges in $E$ incident only to the non eliminated nodes. Finally, the set $A^w$ is obtained by duplicating all edges in $E^w$ with the exception of arcs $(i, w^s)$ and $(w^t,i)$. We describe this procedure in Algorithm \ref{pre_1}. 

We assume that the cost of constructing each node and each edge is not higher than the budget.

\begin{algorithm}[htpb] 
	\caption{Pre-processing I}
	\label{pre_1} 
	\begin{algorithmic} 
		\FOR{$w \in W$}
		\STATE $N^w = N$
		\FOR{$i \in N$}
		\STATE compute the shortest path for the O/D pairs $(w^s,i)$ and $(i,w^t)$
		\IF{the sum of the length of both paths is greater than $u^w$ }
		\STATE $N^w = N^w \setminus \{i \}$
		\STATE $E^w = E^w \setminus \delta(i)$
		\ENDIF
		\ENDFOR
		\STATE $A^w = \{(i,j) \in A :\{i, j\} \in E^w, \, j \neq w^s,\, i\neq w^t\}$
		\ENDFOR
 \RETURN $\{ \mathcal{N}^w = (N^w, E^w), A^w \}_{w\in W}$
\end{algorithmic}
\end{algorithm}

Next, we focus on $(MC)$. We can eliminate O/D pairs $w$ that are too expensive to be covered. That means, the O/D pair $w$ is deleted from $W$ if there is no path between $w^s$ and $w^t$ satisfying: i. its building cost is less than $C_{max}$; and ii. its length is less than $u^w$.\\
This can be checked by solving a shortest path problem with resource constraints and can thus be done in a pseudo-polynomial time. \cite{Desrochers1986algorithm} shows how to adapt Bellman-Ford algorithm to solve it. However, given the moderate size of graphs we consider, we solve it as a feasibility problem. For each $w$, we consider the feasibility problem associated to constraints (\ref{eq:budget}) (\ref{rel_xy}), (\ref{flow_cons}), (\ref{Capacity}), (\ref{Utility}) and (\ref{eq:nat}), with $z^w$ fixed to $1$. If this problem is infeasible, then the O/D pair $w$ is deleted from $W$. Otherwise, there exists a feasible path denoted by Path$_w$. We denote by $(\widetilde{N}^w, \widetilde{E}^w)$ the subgraph of $\mathcal{N}^w$ induced by Path$_w$.

\subsection{Polyhedral properties}

Both formulations $(MC)$ and $(PC)$ involve flow variables $f^w_a$ whose number can be huge when the number of O/D pairs is large. To circumvent this drawback we use a Benders decomposition approach for solving $(MC)$ and $(PC)$. In this subsection, we present properties of the two formulations that allow us to apply such a decomposition in an efficient way.
The first proposition shows that we can relax the integrality constraints on the flow variables $f^w_a$. Let $(MC\_R)$ and $(PC\_R)$ denote the formulations $(MC)$ and $(PC)$ in which constraints (\ref{eq:nat-f}) are replaced by non-negativity constraints, i.e.
\begin{equation}
f^w_a \geq 0, w \in W, a \in A. \label{eq:nneg-f}
\end{equation}
We denote the set of feasible points to a formulation $F$ by $\mathcal{F}(F)$. Further, let $Q$ be a set of points $(\boldsymbol{x},\boldsymbol{z}) \in R^q \times R^p$. 
Then the projection of $Q$ onto the 
$x$-space, denoted $Proj_xQ$, is the set of points given by $Proj_{\boldsymbol{x}} Q = \{\boldsymbol{x} \in R^q: (\boldsymbol{x},\boldsymbol{z}) \in Q$ for some $\boldsymbol{z} \in R^p\}$.

\begin{prop} \label{prop_f}
$Proj_{\boldsymbol{x},\boldsymbol{y},\boldsymbol{z}}(\mathcal{F}(MC)) = Proj_{\boldsymbol{x},\boldsymbol{y},\boldsymbol{z}}(\mathcal{F}(MC\_R))$ and $Proj_{\boldsymbol{x},\boldsymbol{y},\boldsymbol{z}}(\mathcal{F}(PC)) = Proj_{\boldsymbol{x},\boldsymbol{y},\boldsymbol{z}}(\mathcal{F}(PC\_R))$.
\end{prop}

\begin{proof} 

We provide the proof for $(MC)$, the other one being identical. 

First, $\mathcal{F}(MC) \subseteq \mathcal{F}(MC\_R)$ implies $Proj_{\boldsymbol{x},\boldsymbol{y},\boldsymbol{z}}(\mathcal{F}(MC)) \subseteq Proj_{\boldsymbol{x},\boldsymbol{y},\boldsymbol{z}}(\mathcal{F}(MC\_R))$. Second, let $(\boldsymbol{x},\boldsymbol{y},\boldsymbol{z})$ be a point belonging to $Proj_{\boldsymbol{x},\boldsymbol{y},\boldsymbol{z}}(\mathcal{F}(MC\_R))$. For every O/D pair $w \in W$ such that $z^w = 0$ then $\boldsymbol{f}^w=0$. In the case where $z^w=1$, there exists a flow $f^w_a \geq 0$ satisfying (\ref{flow_cons}) and (\ref{Capacity}) that can be decomposed into a convex combination of flows on paths from $w^s$ to $w^t$ and cycles. Given that the flow $f_a^w$ also satisfies (\ref{Utility}), then a flow of value 1 on one of the paths in the convex combination must satisfy this constraint. Hence by taking $f_a^w$ equal to 1 for the arcs belonging to this path and to 0 otherwise, we show that $(\boldsymbol{x},\boldsymbol{y},\boldsymbol{z})$ also belongs to $Proj_{\boldsymbol{x},\boldsymbol{y},\boldsymbol{z}}(\mathcal{F}(MC))$. 
\end{proof}

Note that a similar result is presented in the recent article \cite{ljubic2020benders}. Based on Proposition \ref{prop_f}, we propose a \textit{Benders decomposition} where variables $f^w_a$ are projected out from the model and replaced by {\it Benders feasibility cuts}. As we will see in Section \ref{s:facertdef}, we also consider the \textit{Benders facet-defining cuts} proposed in \cite{Conforti2019facet}. To apply this technique it is necessary to get an interior point of the convex hull of $Proj_{\boldsymbol{x},\boldsymbol{y},\boldsymbol{z}}(\mathcal{F}(MC\_R))$ (resp. $Proj_{\boldsymbol{x},\boldsymbol{y},\boldsymbol{z}}(\mathcal{F}(PC\_R))$). The following property gives us an algorithmic tool to apply this technique to $(MC)$.

\begin{prop} \label{prop_MC}
After pre-processing, the convex hull of $Proj_{\boldsymbol{x},\boldsymbol{y},\boldsymbol{z}}(\mathcal{F}(MC\_R))$ is full-dimensional. 
\end{prop}

\begin{proof}
 To prove the result, we exhibit $\vert N\vert + \vert E\vert + \vert W\vert + 1$ affinely independent feasible points:
	\begin{itemize}
		\item The $0$ vector is feasible.
		\item For each $i\in N$, the points:
		$$ y_{i}=1, \, y_{i'} = 0, \, i' \in N\setminus\{i\},\quad x_e = 0,\, e\in E,\quad z^w =0, \, w\in W.$$
		\item For each $e=\{i,j\}\in E$, the points:
		$$y_k=1, \, k\in e,\, y_k = 0,\, k \in N \setminus \{i,j\},\quad x_e=1,\, x_{e'} = 0,\, e'\in E\setminus\{e\},\quad z^w =0, \, w\in W.$$
		\item For each $w\in W$, the points:
		\begin{gather*}
		y_i=1,\, i \in \widetilde{N}^w,\, y_i=0,\, i \in N\setminus \widetilde{N}^w, \quad	x_e=1, \, e\in \widetilde{E}^w,\, x_e = 0,\, e\in E\setminus\widetilde{E}^w,\\
		z^w=1,\, z^{w'} = 0, \, w'\in W\setminus\{w\}.
		\end{gather*}
	\end{itemize}
	Clearly these points are feasible and affinely independent. Thus the polytope is full-dimensional.
\end{proof}

The proof of Proposition \ref{prop_MC} gives us a way to compute an interior point of the convex hull of $Proj_{\boldsymbol{x},\boldsymbol{y},\boldsymbol{z}}(\mathcal{F}(MC\_R))$. The average of these $|N| + |E| + |W| + 1$ points is indeed such an interior point.

This is not the case for $(PC)$ as we show in Example $1$.

\paragraph{Example 1} Consider the instance of $(PC)$ given by the data presented in Table \ref{t:CP} and Figure \ref{fig:exPC}. We consider the case where at least half of the population must be covered, that is $\beta = 0.5$. In order to satisfy the trip coverage constraint \eqref{eq:PCconst}, the O/D pair $w=(1,4)$ must be covered. Hence $z^{(1,4)} = 1$ is an implicit equality. Furthermore, the only path with a length less than or equal to $u^{(1,4)} = 15$ is composed of edges \{1,2\} and \{2,4\}. Hence, $x_{\{1,2\}}$, $x_{\{2,4\}}$, $y_1$, $y_2$ and $y_4$ must take value $1$. In consequence, the polytope associated to $(PC)$ is not full-dimensional.

\begin{minipage}{0.45\textwidth}
\begin{table}[H]
\centering
\begin{tabular}{cccc}
\hline
 Origin & Destination & $u^w$ & $g^w$ \\
\hline
1 & 4 & 15 & 200 \\
2 & 4 & 10 & 50 \\
3 & 4 & 15 & 50 \\
\hline
\end{tabular}
\caption{Data in Example 1. We consider $\beta = 0.5$.}
\label{t:CP}
\end{table}
\end{minipage}
\begin{minipage}{0.45\textwidth}
\begin{figure}[H]
\centering
\includegraphics[width = 0.8\textwidth]{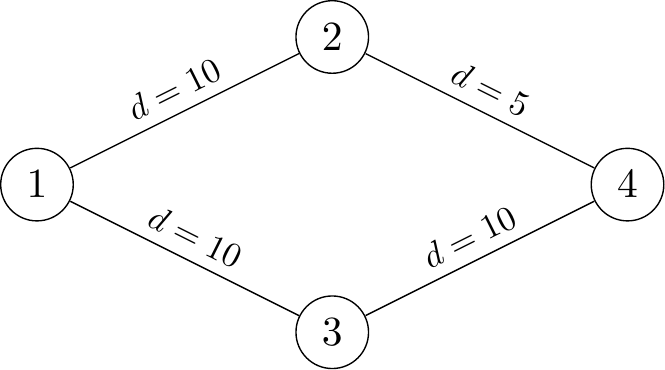}
\caption{Graph of Example 1.}
\label{fig:exPC}
\end{figure}
\end{minipage}
\\

We can compute the dimension of the convex hull of $Proj_{\boldsymbol{x},\boldsymbol{y},\boldsymbol{z}}(\mathcal{F}(PC\_R))$ in an algorithmic fashion. 

We find feasible affinely independent points and, at the same time, we detect O/D pairs which must be covered in any feasible solution. Due to the latter, there are a subset of nodes and a subset of edges that have to be built in any feasible solution. This means that there is a subset of design variables $y_i, i\in N$, $x_e, e\in E$ and mode choice variables $z^w, w\in W$ that must take value $1$. At the opposite to $(MC)$, a solution to $(PC)$ with all variables set to $0$ is not feasible. However, the solution obtained by serving all O/D pairs and building all nodes and edges is feasible. Therefore, we start with a solution with all variables in $\boldsymbol{x},\boldsymbol{y},\boldsymbol{z}$ set to $1$ and we check, one by one, if it is feasible to set them to $0$. By setting one variable $x_e$ or $y_i$ to $0$, it may become impossible to cover some O/D pair $w$. In this case, we say that edge $e$ and node $i$ is \textit{essential} for $w$. To simplify the notation, we introduce the binary parameters $\theta^w_{e}$ and $\theta^w_{i}$ taking value $1$ if edge $e$ (respectively node $i$) is essential for $w$. These new points are stored in a set $L$. Each time the algorithm finds a variable that cannot be set to $0$, we store it in sets $\bar N,\, \bar E,\, \bar W$, respectively. At the end of the algorithm, the dimension of the convex hull of $Proj_{\boldsymbol{x},\boldsymbol{y},\boldsymbol{z}}(\mathcal{F}(PC\_R))$ is 
$$
\dim(\mathcal P_{\boldsymbol{x},\boldsymbol{y},\boldsymbol{z}}) = \vert N\vert + \vert E\vert + \vert W\vert - (\vert \bar{N} \vert + \vert \bar E \vert + \vert \bar{W} \vert).
$$

\noindent This procedure is depicted in Algorithm \ref{pre_2}.

\begin{algorithm}[htpb] 
	\caption{Computing the dimension of the polytope of $(PC)$} 
	\label{pre_2}
	\begin{algorithmic} 
		\STATE {\bf Initialization:} Set $\bar N = \emptyset$, $\bar E = \emptyset$, $\bar W = \emptyset$ and $L = \emptyset$ 
		\STATE Add to set $L$: $\left( y_i = 1, i\in N,\quad x_e = 1, e \in E,\quad z^w = 1, w \in W\right).$
		\FOR{$w' \in W$}
		\IF{$\sum\limits_{w \in W\setminus\{w'\}} g^w \ge \beta\, G$}
		\STATE		
		\STATE Add to set $L$:
		$\left(y_i = 1, i\in N,\quad x_e = 1, e \in E,\quad z^{w'} = 0, \, z^w = 1, w \in W\setminus\{w'\}\right).$
		\ELSE 
		\STATE $\bar W = \bar W \cup \{w'\}$ 
		\FOR{$e=\{i,j\} \in E$}
		\STATE Compute shortest path between $w'^{s}$ and $w'^{t}$ in the graph $(N^{w'}, E^{w'}\setminus\{e\})$.
		\IF{the length of the shortest path is greater than $u^{w'}$ \OR there is no path between $w'^s$ and $w'^t$}
		\STATE $\bar E = \bar E \cup \{e\}$ and $\bar N = \bar N \cup \{i,j\} $
		\ENDIF
		\ENDFOR
		\ENDIF
		\ENDFOR
		\FOR{$e' \in E\setminus \bar{E}$}
		\STATE Add to set $L$:
		$\left(y_i = 1, \, i\in N,\quad x_e = 1, e \in E\setminus \{e'\}, x_{e'}=0,\quad z^w = 1-\theta^w_{e'}, w \in W\right).$
		\ENDFOR
		\FOR{$i' \in N\setminus \bar{N}$}
		\STATE Add to set $L$:
		
		$\bigl(y_{i'}=0, \, y_i = 1, i\in N \setminus \{i'\},\quad x_e = 0, i'\in e, \, x_e = 1, i'\notin e,\quad z^w = 1-\theta^w_{i'}, w\in W\bigr).$
		\ENDFOR
		\STATE $\dim(\mathcal P_{\boldsymbol{x},\boldsymbol{y},\boldsymbol{z}}) = \vert N\vert + \vert E\vert + \vert W\vert - (\vert \bar N \vert + \vert \bar{E} \vert + \vert \bar W\vert).$
		\RETURN $\bar N, \,\bar E,\, \bar W$, $L$ and $\dim(conv(P_{\boldsymbol{x},\boldsymbol{y},\boldsymbol{z}}))$.
	\end{algorithmic}
\end{algorithm}

Algorithm \ref{pre_2} allows : i) to set some binary variables equal to $1$, decreasing the problem size; and ii) to compute a relative interior point of the convex hull of $Proj_{\boldsymbol{x},\boldsymbol{y},\boldsymbol{z}}(\mathcal{F}(PC\_R))$, necessary for the \textit{facet-defining cuts}, as explained below in Section \ref{s:facertdef}. The relative interior point is given by the average of the points in set $L$.

\paragraph{Example 1 cont}
Regarding the previous example and following Algorithm \ref{pre_2}, the O/D pair $(1,4)$ must be covered, $z^{(1,4)}=1$. Due to that, as its shortest path in the networks $(N^{(1,4)}, E^{(1,4)}\setminus\{\{1,2\}\})$ and $(N^{(1,4)}, E^{(1,4)}\setminus\{\{2,4\}\})$ is greater than $u^{(1,4)}=15$, variables $x_{\{1,2\}}$, $x_{\{2,4\}}$, $y_1$, $y_2$, $y_4$ are set to $1$. Finally, the dimension of this polyhedron is
$$\dim(P_{\boldsymbol{x},\boldsymbol{y},\boldsymbol{z}}) = 4+4+3-(3+2+1)=5.$$

The relative interior point computed is:
\begin{align*}
\,x_{\{1,2\}} = 1,\,\,x_{\{2,4\}} = 1, x_{\{1,3\}} = \frac{5}{6}, \,x_{\{3,4\}} = \frac{2}{3},\, y_1 = 1,\, y_2 = 1,\, y_3 = \frac{5}{6},\, y_4 = \frac{5}{6},\\
z^{(1,4)} = 1,\,z^{(2,4)} = \frac{5}{6},\,z^{(3,4)} = \frac{1}{2}.
\end{align*}

\subsection{Setting an initial solution} \label{s:initial_solution}

We determine an initial feasible solution for $(MC)$ and $(PC)$ with a simple greedy heuristic in which we sequentially select O/D pairs with best ratio demand overbuilding cost. More precisely, given the potential network $\mathcal{N} = (N,E)$, we compute for each O/D pair $w$ the ratio $r_w = \frac{g^w}{C(\mbox{Path}_w)}$, where $C(\mbox{Path}_w)$ is the cost of a feasible path for $w$. We order these ratios decreasingly. We use this initial order in the heuristic for both $(MC)$ and $(PC)$.
For $(MC)$ the method proceeds as follows. It starts with an empty list of built nodes and edges, an empty list of O/D pairs covered, and a total cost set to $0$. For each O/D pair $w$, in decreasing order of $r_w$, the heuristic tries to build Path$_w$ considering edges and nodes that are already built. If the additional cost plus the current cost is less than the budget $C_{max}$, nodes and edges in Path$_w$ are built and the O/D pair $w$ is covered (i.e. $z^w = 1$). The total cost, the lists of built nodes and edges are updated. Otherwise we proceed with the next O/D pair. At the end of the algorithm we have an initial feasible solution.

To get an initial solution for $(PC)$ we start with a list of all the O/D pairs covered and the amount of population covered equal to $G$. For each O/D pair $w$, in decreasing order of $r_w$, the algorithm checks if by deleting the O/D pair $w$ from the list, the coverage constraint (\ref{eq:PCconst}) is satisfied. If so, the O/D pair $w$ is deleted from the list and the amount of population covered is updated. Finally, the algorithm builds the union of the subgraphs $(\widetilde{N}^w,\widetilde{E}^w)$ induced by Path$_w$ for all the O/D pairs covered. Note that both initial solutions can be computed by solving $\vert W \vert$ shortest paths problems. These tasks can be executed  much faster  than solving  $(MC)$ and $(PC)$ to optimality.

Pseudo-codes for both routines are provided in \ref{ap:mipstart}. In Section \ref{s:computational}, we will show the efficiency of adding this initial solution at the beginning of the branch-and-Benders-cut procedure.

\section{Benders Implementations}\label{sec:BImp}
In the following, we describe different Benders implementations for $(MC)$ and $(PC)$ obtained by projecting out variables $f^w_a$. Given that $(MC)$ and $(PC)$ share the same subproblem structure the Benders decomposition applied to $(MC)$ is valid for $(PC)$ and vice versa. Thus, we will apply the same Benders decomposition for both problems throughout this manuscript. These implementations are used as sub-routines in a branch-and-Benders-cut scheme. This scheme allows cutting infeasible solutions along the branch-and-bound tree. Depending on the implementation, infeasible solutions can be separated at any node in the branch-and-bound tree or only when an integer solution is found. In the case of $(MC)$, the master problem that we solve is:
\begin{align}
(M\_MC) \quad \max\limits_{\boldsymbol{x},\boldsymbol{y},\boldsymbol{z}} & \quad \sum_{w\in W} g^w z^w \\
\mbox{s.t. } & \quad \mbox{(\ref{eq:budget}), \, (\ref{rel_xy}), \, (\ref{eq:nat})} & \nonumber\\
&\quad +\{\mbox{Benders Cuts }(\boldsymbol{x},\boldsymbol{y},\boldsymbol{z})\}. \nonumber
\end{align}
\noindent The master problem for $(PC)$, named $(M\_PC)$, is stated analogously.

In Section \ref{s:lpfeas}, we discuss the standard \textit{Benders cuts} obtained by dualizing the respective feasibility subproblem. Then, in Section \ref{s:normalben} we discuss ways of generating normalized subproblems, to produce stronger cuts. We name these cuts \textit{normalized Benders cuts}. In Section \ref{s:facertdef}, we apply {\it facet-defining cuts} in order to get stronger cuts, as it is proposed in \cite{Conforti2019facet}. Finally, we discuss an implementation where, at the beginning, \textit{cut-set inequalities} are added to enhance the link between $\boldsymbol{z}$ and $\boldsymbol{x}$, and then {\it Benders cuts} are added.

\subsection{LP feasibility cuts} \label{s:lpfeas}

Since the structure of the model allows it, we consider a feasibility subproblem made of constraints (\ref{flow_cons}), (\ref{Capacity}), (\ref{Utility}) and (\ref{eq:nneg-f}) for each commodity $w\in W$, denoted by $(SP)^w$. We note that each subproblem is feasible whenever $z^w=0$, so it is necessary $(SP)^w$ to check feasibility only in the case where $z^w>0$. As it is clear from the context, we remove the index $w$ from the notation. The dual of each feasibility subproblem can be expressed as: 
\begin{align}
(DSP)^w\quad \max\limits_{\boldsymbol{\alpha},\boldsymbol{\sigma},\boldsymbol{\upsilon}}&\quad z\,\alpha_{w^s} - \sum_{e \in E} x_e\,\sigma_e - u\, z\,\upsilon\\
\mbox{s.t. }&\quad \alpha_i - \alpha_j - \sigma_e - d_a\, \upsilon \le 0,& a=(i,j) \in A : e=\{i,j\}, \\
&\quad \sigma_e, \, \upsilon \ge 0, & e\in E,
\label{sub:trd}
\end{align}
\noindent where $\boldsymbol{\alpha}$ is the vector of dual variables related to constraints (\ref{flow_cons}), $\boldsymbol{\sigma}$ is the vector of dual variables corresponding to the set of constraints (\ref{Capacity}) and $\boldsymbol{\upsilon}$ is the dual variable of constraint (\ref{Utility}). Since constraints in (\ref{flow_cons}) are linearly dependent, we set $\alpha_{w^t}=0$. Given a solution of the master problem $(\boldsymbol{x},\boldsymbol{y},\boldsymbol{z})$, there are two possible outcomes for $(SP)^w$:
\begin{enumerate}
\item $(SP)^w$ is infeasible and $(DSP)^w$ is unbounded. Then, there exists an increasing direction $(\boldsymbol{\alpha},\boldsymbol{\sigma},\boldsymbol{\upsilon})$ with positive cost. In this case,
the current solution $(\boldsymbol{x},\boldsymbol{y},\boldsymbol{z})$ is cut by:
\begin{equation}
 (\alpha_{w^s} - u\,\upsilon)\,z - \sum_{e \in E} \sigma_e\, x_e \le 0. \label{eq:feasibility_cut}
\end{equation}
\item $(SP)^w$ is feasible and consequently, $(DSP)^w$ has an optimal objective value equal to zero. In this case, no cut is added. 
\end{enumerate}

\subsection{Normalized Benders cuts} \label{s:normalben}

The overall branch-and-Benders-cut performance heavily relies on how the cuts are implemented. It is known that feasibility cuts may have poor performance due to the lack of ability of selecting a {\it good} extreme ray (see for example \cite{Fischetti2010note, Ljubic2012exact}). However, normalization techniques are known to be efficient to overcome this drawback \cite{Magnanti1981accelerating, balas2002lift,balas2003precise}. The main idea is to transform extreme rays in extreme points of a suitable polytope. In this section we study three ways to normalize the dual subproblem described above.
 
 First, we note that the feasibility subproblem can be reformulated as a min cost flow problem in $\mathcal{N}^w$ with capacities $\boldsymbol{x}$ and arc costs $d_a$.
 \begin{align}
(NSP)^w \quad  \min\limits_{\boldsymbol{x},\boldsymbol{y},\boldsymbol{z},\boldsymbol{f}} & \quad \sum_{a\in A} d_a\, f_a & \\
\mbox{s.t. } & \quad \mbox{(\ref{flow_cons}), (\ref{Capacity}), (\ref{eq:nneg-f}).} & \nonumber 
\end{align}
The associated dual subproblem is:
\begin{align}
(DNSP)^w \quad \max\limits_{\boldsymbol{\alpha},\boldsymbol{\sigma}} & \quad z\,\alpha_{w^s} - \sum_{e \in E} \sigma_e x_e \\
\mbox{s.t. } & \quad  \alpha_i - \alpha_j - \sigma_e \le d_a,& a=(i,j) \in A : e=\{i,j\}, \\
&\quad\sigma_e \ge 0,& e \in E.
\end{align}
Whenever $z^w>0$, the primal subproblem $(NSP)^w$ may be infeasible. Subproblems $(NSP)^w$ are no longer feasibility problems, although some of their respective dual forms can be unbounded. As the  splitting demand constraint has to be satisfied there are two kind of cuts to add:
\begin{enumerate}
 \item $(NSP)^w$ is infeasible and $(DNSP)^w$ is unbounded. In this case, the solution $(\boldsymbol{x},\boldsymbol{y},\boldsymbol{z})$ is cut by the constraint
 \begin{equation}
 \alpha_{w^s}\, z - \sum\limits_{e\in E}\sigma_e \, x_e \leq 0.
 \label{cut:norm1}
 \end{equation}
 \item $(NSP)^w$ is feasible and $(DNSP)^w$ has optimal solution. Consequently, if their solutions $(\boldsymbol{\alpha},\boldsymbol{\sigma})$ and $(\boldsymbol{x},\boldsymbol{y},\boldsymbol{z})$ satisfy that $\alpha_{w^s}\, z - \sum_{e\in E}\sigma_e \, x_e > u\, z$ then, the following cut is added
 \begin{equation}
 (\alpha_{w^s} - u)\, z - \sum\limits_{e\in E}\sigma_e \, x_e \leq 0.\label{cut:norm2}
 \end{equation}
\end{enumerate}

\noindent We refer to this implementation as \texttt{BD\_Norm1}. 

In this situation, there still exists dual subproblems $(DNSP)^w$ with extreme rays. We refer to \texttt{BD\_Norm2} as second dual normalization obtained by adding the dual constraint $\alpha_{w^s} = u + 1$. In this case, every extreme ray of $(SP)^w$ corresponds to one of the extreme points of $(NSP)^w$. A cut is added whenever the optimal dual objective value is positive. This cut has the following form:

\begin{equation}
z - \sum_{e \in E} \sigma_e\, x_e \le 0. \label{eq:normalized_cut} 
\end{equation}

We finally tested a third dual normalization, \texttt{BD\_Norm3}, by adding constraints
\begin{equation}
\sigma_e \leq 1,\qquad e \in E, \label{eq:normalized_cut2} 
\end{equation}
\noindent directly in $(DSP)^w$.

We tested the three dual normalizations described above for $(MC)$ using randomly generated networks with $10$, $20$ and $40$ nodes, as described in Subsection \ref{datasets}. As we will see in Section \ref{s:preliminary_results}, only \texttt{BD\_Norm1} results to be competitive.

\subsection{Facet-defining Benders cuts} \label{s:facertdef}

Here we describe how to generate Benders cuts for $(MC)$ based on the ideas exposed in \cite{Conforti2019facet}, named as $CW$. The procedure for $(PC)$ is the same. Given an \textit{interior point} or \textit{core point} $(\boldsymbol{x}^{in}, \boldsymbol{y}^{in}, \boldsymbol{z}^{in})$ of the convex hull of feasible solutions and an \textit{exterior point} $(\boldsymbol{x}^{out}, \boldsymbol{y}^{out}, \boldsymbol{z}^{out})$, that is a solution of the LP relaxation of the current restricted master problem, a cut that induces a facet or an improper face of the polyhedron defined by the LP relaxation of $Proj_{\boldsymbol{x},\boldsymbol{y},\boldsymbol{z}}\mathcal{F}(MC)$ is generated. We denote the difference $\boldsymbol{x}^{out} - \boldsymbol{x}^{in}$ by $\Delta \boldsymbol{x}$. We define $\Delta \boldsymbol{y}$ and $\Delta \boldsymbol{z}$ analogously. The idea is to find the furthest point from the core point, feasible to the LP-relaxation of $Proj_{\boldsymbol{x},\boldsymbol{y},\boldsymbol{z}}\mathcal{F}(MC)$ and lying on the segment line between the \textit{core point} and the \textit{exterior point}. This point is of the form $(\boldsymbol{x}^{sep}, \boldsymbol{y}^{sep}, \boldsymbol{z}^{sep}) = (\boldsymbol{x}^{out}, \boldsymbol{y}^{out}, \boldsymbol{z}^{out}) - \lambda (\Delta \boldsymbol{x}, \Delta \boldsymbol{y} , \Delta \boldsymbol{z})$. The problem of generating such a cut reads as follows:
\begin{align}
(SP\_CW)^w \quad \min_{\boldsymbol{f},\lambda} & \quad \lambda \\
\mbox{s.t. }& \quad \sum_{a \in \delta_w^+(i)} f_a -\sum_{a \in \delta_w^-(i)} f_a = 
 \begin{cases}
z^{out} - \lambda\,\Delta z, & \text{if $i = w^s$,}\\
0, & \text{otherwise,} 
\end{cases} \\
&\quad  f_a + f_{\hat{a}} \le x^{out}_e - \lambda\, \Delta x_e, \quad\qquad e=\{i,j\} \in E : a=(i,j), \hat{a}=(j,i),\\
& \quad \sum_{a \in A} d_a\, f_a \le u\,z^{out} - u\,\Delta z\, \lambda, \\
&\quad  0\le \lambda \le 1, \\
&\quad  f_a \ge 0,\hspace{8cm} a \in A.
\end{align}
In order to obtain the Benders feasibility cut we solve its associated dual:
\begin{align}
(DSP\_CW)^w \quad \max_{\boldsymbol{\alpha}, \boldsymbol{\sigma},\boldsymbol{\upsilon}} &\quad z^{out}\,\alpha_{w^s} - \sum_{e \in E} x^{out}_e\,\sigma_e - u\,z^{out}\, \upsilon \\
\mbox{s.t. }&\quad \Delta z\, \alpha_{w^s} - \sum_{e \in E} \Delta x_e\,\sigma_e - u\, \Delta z\, \upsilon \le 1, \label{eq:normcw}\\
&\quad \alpha_i - \alpha_j -\sigma_e - d_a\,\upsilon \le 0,\qquad\qquad\quad a=(i,j) \in A : e=\{i,j\},\nonumber\\
&\quad \sigma_e,\,\upsilon\ge 0, \hspace{6.9cm} e\in E. \nonumber
\label{sub:CW}
\end{align}
Given that $(SP\_CW)^w$ is always feasible ($\lambda = 1$ is feasible) and that its optimal value is lower bounded by 0, then, both $(SP\_CW)^w$ and $(DSP\_CW)^w$ have always finite optimal solutions. Whenever the optimal value of $\lambda$ is 0, $(\boldsymbol{x}^{out}, \boldsymbol{y}^{out}, \boldsymbol{z}^{out})$ is feasible. A cut is added if the optimal value of $(DSP\_CW)^w$ is strictly greater than 0. The new cut has the same form as in (\ref{eq:feasibility_cut}). Note that this problem can be seen as a dual normalized version of $(SP)^w$ with the dual constraint \eqref{eq:normcw}. This approach is an improvement in comparison with the stabilization cuts proposed by \cite{ben2007acceleration}, where $\lambda$ is a fixed parameter.

Core points for both formulations can be obtained by computing the average of the points described in the proof of Proposition \ref{prop_MC} for $(MC)$ and the average of the points in list $L$ obtained by applying Algorithm~\ref{pre_2}.

\subsection{Cut-set inequalities} \label{s:cutset}

By projecting out variable vector $\boldsymbol{f}$, information regarding the link between vectors $\boldsymbol{x}$ and $\boldsymbol{z}$ is lost. {\it Cut-set inequalities} represent the information lost regarding the connectivity for the O/D pair $w$ in the solution given by the design variable vector $\boldsymbol{x}$. Let $(S,S^C)$ be a $(w^s,w^t)$-partition of $N^w$ for a fixed O/D pair $w$, i.e. $(S,S^C)$ satisfies: i. $w^s\in S$; ii. $w^t \in S^C$, with $S^C=N\setminus S$ its complement. A {\it cut-set inequalities} is defined as
\begin{equation}
z^w \leq \sum_{\substack{\{i,j\} \in E^w: \\ i\in S, \, j \in S^C}} x_{\{i,j\}}, \quad w \in W,\quad (S,S^C) \mbox{ a } (w^s,w^t)\mbox{-partition of } N^w. \label{eq:cutset}
\end{equation}
\noindent This type of constraints has been studied in several articles, for instance \cite{Barahona1996network,Koster2013extended, Costa2009benders}. Note that it is easy to see that cut-set inequalities belong to the LP-based Benders family. Let $(S,S^C)$ be a $(w^s,w^t)$-partition in the graph $\mathcal{N}^w$ for $w\in W$. Consider the following dual solution:

\begin{itemize}
	\item $\alpha_i = 1$ if $i\in S$; $\alpha_i = 0$ if $i\in S^C$.
	\item $\sigma_e=1$ if $e=\{i,j\} \in E^w$, $i\in S$, $j \in S^C$; $\sigma_e=0$, otherwise.
	\item $\upsilon = 0$.
\end{itemize}

This solution is feasible to $(DSP)^w$ and induces a cut as in (\ref{eq:cutset}). In order to improve computational performance, we test two approaches to include these inequalities:

\begin{enumerate}
	\item We implement a modification of the Benders callback algorithm with the following idea. First, for each $w\in W$, using the solution vector $(\boldsymbol{x},\boldsymbol{y})$ from the master, the algorithm generates a network $(N^w,E^w)$ with capacity $1$ for each edge built. Then, a \textit{Depth-First Search (DFS)} algorithm is applied to obtain the connected component containing $w^s$. If the connected component does not contain $w^t$, a cut of the form (\ref{eq:cutset}) is added. Otherwise, we generate a Benders cut as before. This routine is depicted in Algorithm \ref{cut_set_implment}.

	\begin{algorithm}[htpb] 
		\caption{Callback implementation with cut-set inequalities.} 
		\label{cut_set_implment} 
		\begin{algorithmic} 
			\REQUIRE ($x_e, e\in E$, $z^w, w\in W$) from the master solution solution $(\boldsymbol{x},\boldsymbol{y},\boldsymbol{z})$.
			\FOR{$w \in W$}
				\STATE Build graph $(N^w({\boldsymbol x}),E^w({\boldsymbol x}))$ induced by the solution vector $\boldsymbol{x}$ from the master.
					
				\STATE Compute the connected component $S$ in $(N^w({\boldsymbol x}),E^w({\boldsymbol x}))$ containing $w^s$.
				\IF{$w^t$ is not included in $S$}
					\STATE Add the cut $
							z^w \leq \sum_{\substack{\{i,j\} \in E^w: \\ i\in S, \, j \in S^C}} x_{\{i,j\}}
							$
				\ELSE
					\STATE Solve the corresponding subproblem ($(DSP)^w$, $(DNSP)^w$, $(DSP\_CW)^w$) and add cut if it is necessary.
				\ENDIF
				\ENDFOR
			\RETURN Cut.
		\end{algorithmic}\label{algorithm:3}
	\end{algorithm}

We tested this implementation with subproblems $(DSP\_CW)^w$. We observe that by using Algorithm \ref{algorithm:3} with $CW$ the convergence is slower and we generate more cuts. These preliminary results are shown in Table \ref{t:pre_res_cs}.

	\item We add to the \textit{Master Problem} the \textit{cut-set inequalities} at the origin and at the destination of each O/D pair $w\in W$ at the beginning of the algorithm. These valid inequalities have the form:
	
	\begin{gather}
	\begin{cases}
	z^w \leq \sum\limits_{e \in \delta(w^s)} x_e ,\\
	z^w \leq \sum\limits_{e \in \delta(w^t)} x_e .
	\end{cases}\label{eq:cutset2}
	\end{gather}

	This means that for each O/D pair to be covered, there should exist at least one edge incident to its origin and one edge incident to its destination, i.e. each O/D pair should have at least one arc going out of its origin and another one coming in its destination.
\end{enumerate}

\section{Computational Results} \label{s:computational}

In this section, we compare the performance of the different families of \textit{Benders cuts} presented in Section \ref{sec:BImp} using the branch-and-Benders-cut algorithm (denoted as \texttt{B\&BC}).

All our computational experiments were performed on a computer equipped with a Intel Core i$5$-$7300$ CPU processor, with $2.50$ gigahertz $4$-core, and $16$ gigabytes of RAM memory. The operating system is $64$-bit Windows $10$. Codes were implemented in Python 3.8. These experiments have been carried out through \texttt{CPLEX 12.10} solver, named \texttt{CPLEX}, using its Python interface. \texttt{CPLEX} parameters were set to their default values and the models were optimized in a single threaded mode.

For that, \texttt{t} denotes the average value for solution times given in seconds, \texttt{gap} denotes the average of relative optimality gaps in percent (the relative percent difference between the best solution and the best bound obtained within the time limit), \texttt{LP gap} denotes the average of LP gaps in percent and \texttt{cuts} is the average of number of cuts generated.

\subsection{Data sets: benchmark networks and random instances}\label{datasets}

We divide the tested instances into two groups: {\it benchmarks instances} and {\it random instances}. Our {\it benchmarks instances} are composed by the Sevilla \citep{garcia2013grasp} and Sioux Falls networks \citep{sioux}. 

The Sevilla instance is composed partially by the real data given by the authors of \cite{garcia2013grasp}. From this data, we have used the topology of the underlying network, cost and distance vector for the set of arcs and the demand matrix. This network is composed of $49$ nodes and $119$ edges. Originally, the set of O/D pairs $W$ was formed by all possible ones ($49\cdot48=2352$). However, some entries in the demand matrix of this instance are equal to 0 and we thus exclude them from the analysis. Specifically, $630$ pairs have zero demand, almost the $27\%$ of the whole set. We consider a private utility $u$ equal to twice the shortest path length in the underlying network. Each node cost is generated according to a uniform distribution $\mathcal{U}(7,13)$. The available budget has been fixed to $30\%$ of the cost of building the whole underlying network and the minimum proportion of demand to be covered to $\beta=0.5$.

For the Sioux Falls instance, the topology of the network is described by $24$ nodes and $38$ edges. Set $W$ is also formed by all possible O/D pairs ($38\cdot37 = 1406$). The parameters have been chosen in the same manner as for the \textit{random instances}.

We generate our {\it random instances} as follows. We consider planar networks with a set of $n$ nodes, with $n \in \{10,20,40,60\}$. Nodes are placed in a grid of $n$ square cells, each one of $10$ units side. For each cell, a point is randomly generated close to the center of the cell.
For each setting of nodes we consider a planar graph with its maximum number of edges, deleting each edge with probability 0.3. We replicated this procedure $10$ times for each $n$, so that the number of nodes is the same while the number of edges may vary. Therefore, there are $40$ different underlying networks. We name these instances as $N10$, $N20$, $N40$ and $N60$. We provide the average cycle availability, connectivity and density for {\it random instances} networks in Table \ref{tabla:indicators}. A couple of them are depicted in Figure \ref{fig:N11N22}.

\begin{figure}[H]
\centering
\begin{tabular}{C{0.4\textwidth}C{0.4\textwidth}}
\includegraphics[height = 3cm]{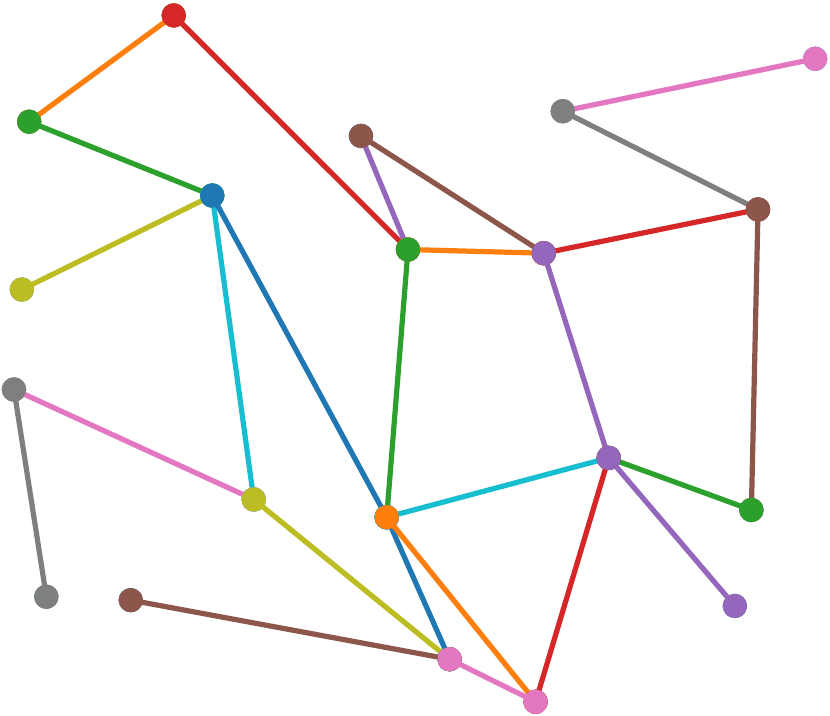} &
\includegraphics[height = 3cm]{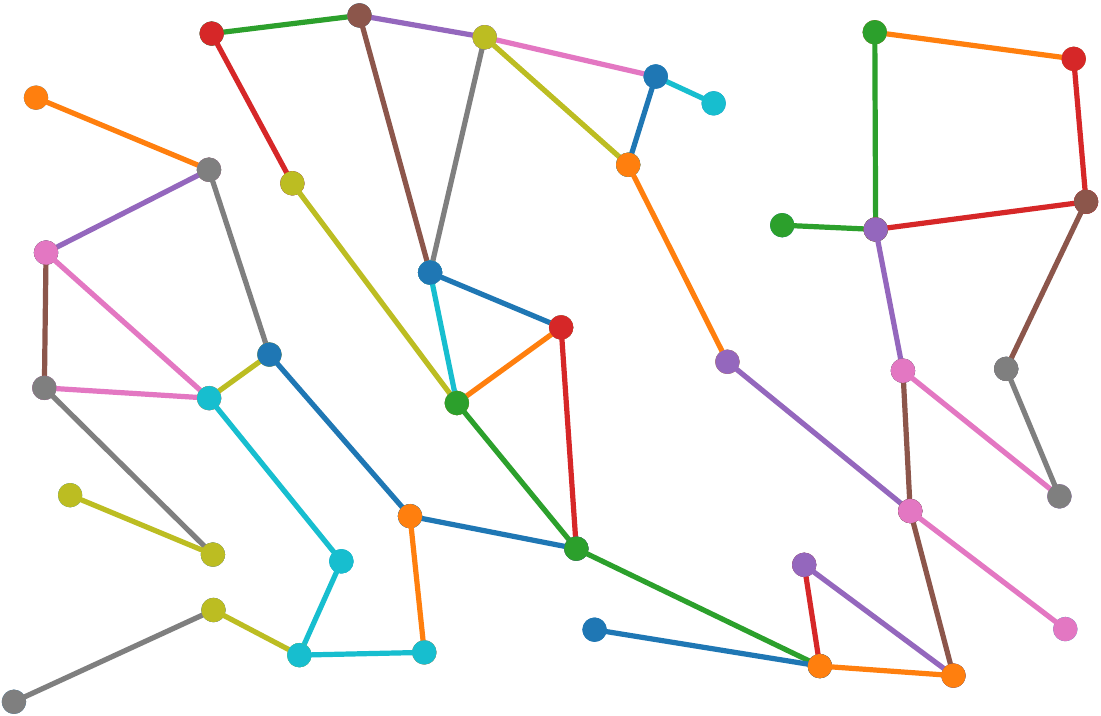}
\end{tabular}

\caption{Example of underlying networks with $|N|$=20 and $|N|$=40.}
\label{fig:N11N22}
\end{figure}

\begin{table}[htb]
\small
\centering
\begin{tabular}{cccc}
\hline
 Network & Cycle availability & Connectivity & Density \\
 & $\frac{|E|-|N|+1}{2|N|-5}$ & $\frac{|E|}{|N|}$ & $\frac{|E|}{3(|N|-2)}$ \\
\hline
N10 & $0.11$ & $1.05$ & $0.44$ \\		
N20 & $0.11$ & $1.12$ & $0.41$ \\		
N40 & $0.13$ & $1.22$ & $0.43$ \\		
N60 & $0.16$ & $1.29$ & $0.45$ \\		
\hline
Overall & $0.12$ & $1.17$ & $0.43$ \\
\hline
 
\end{tabular}
\caption{Cycle availability, connectivity and density parameters for the underlying networks in {\it random instances}.}
\label{tabla:indicators}
\end{table}

Construction costs $b_i$, $i\in N$, are randomly generated according to a uniform distribution $\mathcal{U}(7,13)$. So, each node costs 10 monetary units in average. Construction cost of each edge $e \in E$, $c_e$, is set to its Euclidean length. This means that building the links cost 1 monetary unit per length unit. The node and edge costs are rounded to integer numbers. We set $C_{max}$ equal to $50\%$ of the cost of building the whole underlying network considered. We denote this total cost as $TC$, so $C_{max} = 0.5 \, TC$.

To build set $W$, we randomly pick each possible O/D pair of nodes with probability 0.5. In consequence, this set has $\frac{n(n-1)}{2}$ pairs in average. Parameter $u^w$ is set to 2 times the length of the shortest path between $w^s$ and $w^t$, named as $SPath^w$. Finally, the demand $g^w$ for each O/D pair $w$ is randomly generated according to the uniform distribution $\mathcal{U}(10,300)$.

\subsection{Preliminary experiments} \label{s:preliminary_results}

Before presenting an extensive computational study of the algorithms, we provide some preliminary results to: i. analyze the efficiency of the formulation presented in \cite{garcia2013grasp}; ii. the efficiency of the cut normalizations described in Section \ref{s:normalben} and, iii. the performance of the cut-set based Branch-and-cut procedure described in Section \ref{s:cutset}. 

We first show that our formulation using \eqref{Capacity}-\eqref{Utility} is not only stronger than the one formulated with \eqref{eq:cap_ardila}-\eqref{eq:uti_ardila} but also more efficient.  Table \ref{table:formulations} shows some statistics for the two formulations discussed at the end of Section \ref{s:formulations}, for instances with 10 and 20 nodes. We also tested instances with 40 nodes but most of them were not solved to optimality within one hour. In that case, we provide the optimality gap instead of the solution time. We consider 5 instances of each size. Note that constraints \eqref{eq:uti_ardila} are equivalent to constraints \eqref{Utility} by setting $M=0$. We tested several positive values for $M$.

\begin{table}[H]
\small
\centering
\begin{tabular}{c cc cc}
\hline
\multirow{2}{*}{Network} & \multicolumn{2}{c}{\texttt{Formulation using \eqref{Capacity}-\eqref{Utility}}} & \multicolumn{2}{c}{\texttt{Formulation using \eqref{eq:cap_ardila}-\eqref{eq:uti_ardila}}} \\ \cline{2-5} 
 & \texttt{t} & \texttt{LP gap} & \texttt{t} &\texttt{LP gap} \\ \hline
N10 & 0.17 & 43.21 & 0.26 & 96.43 \\ 
N20 & 5.78 & 56.33 & 228.22 & 106.71 \\ 
\hline
 & \texttt{gap} & \texttt{LP gap} & \texttt{gap} &\texttt{LP gap}\\ 
\hline
N40 & 11.74 & 68.15 & 54.85 & 137.13 \\ \hline
\end{tabular}\caption{Comparing the performance of the two different types of mode choice and capacity constraints for $(MC)$ within a time limit of 1 hour. The majority of N40 instances were not solved to optimality, then the average gap is shown.}
\label{table:formulations}
\end{table}

Secondly, we tested the three dual normalizations described in Section \ref{s:normalben} for $(MC)$. Table \ref{table: algorithm_CS} shows average values obtained for solution time in seconds and number of cuts needed for this experiment. The only one that seems competitive is \texttt{BD\_Norm1}. We observed that cut coefficients generated with \texttt{BD\_Norm1} are mainly 0's or 1's. In the case of \texttt{BD\_Norm2} and \texttt{BD\_Norm3} we observe that coefficients generated are larger than the ones generated by \texttt{BD\_Norm1}, so they may induce numerical instability. This situation is similar for the case of $(PC)$.

\begin{table}[htpb]
\small
\centering
\begin{tabular}{c cc cc cc}
\hline
\multirow{2}{*}{Network} & \multicolumn{2}{c}{\texttt{BD\_Norm1}} & \multicolumn{2}{c}{\texttt{BD\_Norm2}} & \multicolumn{2}{c}{\texttt{BD\_Norm3}} \\ \cline{2-7} 
 & \texttt{t} & \texttt{cuts} & \texttt{t} &\texttt{cuts} & \texttt{t} & \texttt{cuts} \\ \hline
N10 & 0.21 & 44 & 0.22 & 47 & 0.24 & 104 \\ 
N20 & 2.83 & 362 & 5.76 & 595 & 5.22 & 1418 \\ 
N40 & 687.88 & 2904 & * & * & * & * \\ \hline
\end{tabular}\caption{Comparing the performance of the three dual normalizations within a time limit of 1 hour for $(MC)$. N10, N20 and N40 are refereed to networks with 10, 20 and 40 nodes, respectively. The mark '*' indicates that four over five instances were not solved within 1 hour.}
\label{table: algorithm_CS}
\end{table}

Finally, we tested the cut-set inequalities implementation described in Section \ref{s:cutset} with subproblems $(DSP\_CW)^w$. We observe that by using Algorithm \ref{algorithm:3} with $CW$ the convergence is slower and we generate more cuts. This might be due to the fact that these cuts do not include information about the length of the path in the graph, but only information regarding the existence of the path. These preliminary results are shown in Table \ref{t:pre_res_cs}, which provides average values obtained for solution times in seconds and the number of cuts added.

\begin{table}[H]
\small
\centering
\begin{tabular}{c cc cc}
\hline
\multirow{2}{*}{Network} & \multicolumn{2}{c}{\texttt{BD\_CW}} & \multicolumn{2}{c}{\texttt{Algorithm \ref{algorithm:3}+BD\_CW}} \\ \cline{2-5} 
 & \texttt{t} & \texttt{cuts} & \texttt{t} & \texttt{cuts} \\ \hline
N10 & 0.23 & 48 & 0.15 & 46 \\ 
N20 & 2.47 & 411 & 2.53 & 500\\ 
N40 & 619.31 & 3486 & 722.02 & 3554 \\ \hline
\end{tabular}\caption{Comparing the performance of the Algorithm \ref{algorithm:3} for $(MC)$. N10, N20 and N40 refer to networks with 10, 20 and 40 nodes respectively.}
\label{t:pre_res_cs}
\end{table}

In conclusion, all these three implementations, with the exception of \texttt{BD\_norm1}, are excluded from further analysis.

\subsection{Branch-and-Benders-cut performance}

Our preliminary experiments show that including cuts only at integer nodes of the branch and bound tree is more efficient than including them in nodes with fractional solutions. Thus, in our experiments we only separate integer solutions unless we specify the opposite. We used the \texttt{LazyCons\-traintCallback} function of \texttt{CPLEX} to separate integer solutions. Fractional solutions were separated using the \texttt{UserCutCallback} function. We study the different implementations of \texttt{B\&BC} proposed in Sections \ref{s:lpfeas}, \ref{s:normalben} and \ref{s:facertdef}. We use the following nomenclature:
\begin{itemize}
	\item \texttt{BD\_Trd}: \texttt{B\&BC} algorithm using the feasibility subproblems structure $(DSP)^w$, and its corresponding feasibility cuts \eqref{eq:feasibility_cut}.
	\item \texttt{BD\_Norm}: \texttt{B\&BC} algorithm using the normalized subproblems structure $(DNSP)^w$, and its corresponding cuts \eqref{cut:norm1} and \eqref{cut:norm2}.
	\item \texttt{BD\_CW}: \texttt{B\&BC} algorithm using the subproblems structure $(DSP\_CW)^w$, and feasibility cuts \eqref{eq:feasibility_cut}. 
\end{itemize}

We compare our algorithms with the direct use of \texttt{CPLEX}, and the automatic Benders procedure proposed by \texttt{CPLEX}, noted by \texttt{AUTO\_BD}. \texttt{CPLEX} provides different implementations depending on the information that the user provides to the solver: i. \texttt{CPLEX} attempts to decompose the model strictly according to the decomposition provided by the user; ii. \texttt{CPLEX} decomposes the model by using this information as a hint and then refines the decomposition whenever possible; iii. \texttt{CPLEX} automatically decomposes the model, ignoring any information supplied by the user. We have tested these three possible settings, and only the first one is  competitive. 

Furthermore we have tested the following features:

\begin{itemize}
\item \texttt{CS}: If we include {\it cut-set inequalities} at each origin and destination as in \eqref{eq:cutset2}.
\item \texttt{IS}: If we provide an initial solution to the solver.
\item \texttt{RNC}: If we add Benders cuts at the root node.
\end{itemize}

\subsection{Performance of the algorithms on random instances}

All the experiments have been performed with a limit of one hour of CPU time considering $10$ instances of each size. Tables in this section show average values obtained for solution times in seconds, relative gaps in percent, and number of cuts needed. To determine these averages, we only consider the instances solved at optimality by all the algorithms. 

First, we compare the performance of \texttt{CPLEX} for formulations $(MC)$ and $(PC)$ and the three different \texttt{B\&BC} implementations described above (\texttt{BD\_Trd}, \texttt{BD\_Norm} and \texttt{BD\_CW}). We also study the impact of the initial cut set inequalities (\texttt{CS}) in the efficiency of the proposed algorithms. Table \ref{table:perfN102040} shows the performance of the algorithms for networks N10, N20 and N40.  All the algorithms are able to solve at optimality N10 and N20 instances  in less than 7 seconds for $(MC)$ and $(PC)$. For $(MC)$ without \texttt{CS}, the fastest algorithm is \texttt{BD\_CW} in sets N10, N20 and N40 for the instances solved at optimality. This is not the case for $(PC)$, since we can observe that \texttt{AUTO\_BD} is slightly faster. In general, when \texttt{CS} is included, the solution time and the amount of cuts required decrease. Specifically, for $(MC)$ in N40, the most efficient algorithm is \texttt{BD\_CW+CS} which gets the optimal solution $43.8\%$ faster than \texttt{Auto\_BD+CS}. For $(PC)$, it seems to be also profitable, since for N40 \texttt{BD\_CW+CS} gets the optimal solution using $55\%$ less time than \texttt{Auto\_BD}. These results are shown in the second and fourth block of Table \ref{table:perfN102040}.

\begin{table}[htpb]
\footnotesize
\centering
\begin{tabular}{c c c c cc cc cc cc}
\cline{2-12}
&&\multirow{2}{*}{Network} & \texttt{CPLEX} &\multicolumn{2}{c}{\texttt{Auto\_BD}} & \multicolumn{2}{c}{\texttt{BD\_Trd}} & \multicolumn{2}{c}{\texttt{BD\_Norm}} & \multicolumn{2}{c}{\texttt{BD\_CW}} \\
\cline{4-12} 
&& & \texttt{t} & \texttt{t} & \texttt{cuts} & \texttt{t} & \texttt{cuts} & \texttt{t} & \texttt{cuts} & \texttt{t} & \texttt{cuts} \\ \hline
\multirow{6}{*}{$(MC)$} &\multirow{3}{0.5cm}{\centering{w.o. \texttt{CS}}} & N10 & 0.18 & 0.43 & 27 & 0.25 & 92 & 0.24 & 91 & 0.19 & 94 \\ 
&& N20 & 6.77 & 4.51 & 273 & 3.89 & 620 & 3.18 & 590 & 3.34 & 641 \\ 
&& N40 & 1646.93 & 617.85 & 1967 & 1095.25 & 3990 & 541.03 & 3677 & 457.81 & 4137 \\ \cline{2-12} 
&\multirow{3}{*}{\texttt{+CS}}& N10 & - & 0.32 & 12 & 0.21 & 49 & 0.28 & 52 & 0.23 & 54 \\ 
&& N20 & - & 3.94 & 178 & 2.29 & 382 & 2.50 & 383 & 1.85 & 416 \\
&& N40 & - & 484.95 & 1248 & 637.49 & 2378 & 575.87 & 2530 & 272.39 & 3186 \\ \hline
\multirow{6}{*}{$(PC)$} &\multirow{3}{0.5cm}{\centering{w.o. \texttt{CS}}} & N10 & 0.18 & 0.29 & 16 & 0.24 & 92 & 0.28 & 89 & 0.20 & 91 \\ 
 && N20 & 6.73 & 4.87 & 305 & 3.55 & 607 & 4.68 & 681 & 2.15 & 606 \\ 
&& N40 & 2153.15 & 504.06 & 1752 & 657.59 & 4470 & 514.42 & 4246 & 837.41 & 4412 \\  \cline{2-12} 
&\multirow{3}{*}{\texttt{+CS}}& N10 & - & 0.28 & 11 & 0.16 & 56 & 0.20 & 57 & 0.145 & 54 \\
&& N20& - & 4.12 & 213 & 3.11 & 497 & 3.43 & 495 & 2.070 & 461 \\
&& N40 & - & 439.23 & 1527 & 261.74 & 3528 & 323.21 & 3583 & 197.55 & 3949 \\ \hline
\end{tabular}
\caption{Comparing the performance of the three algorithms for $(MC)$ and $(PC)$.}
\label{table:perfN102040}
\end{table}

Table \ref{table:optN40} shows the instances in N40 solved in one hour.   Without \texttt{CS}, some instances in set N40 cannot be solved to optimality neither for $(MC)$ nor for $(PC)$. Nevertheless, by including \texttt{CS}, Benders implementations can solve all the instances in N40 in the one hour limit.

\begin{table}[htpb]
\small
\centering
\begin{tabular}{ccccccc}
\cline{2-7}
& & \texttt{CPLEX}& \texttt{Auto\_BD} & \texttt{BD\_Trd} & \texttt{BD\_Norm} & \texttt{BD\_CW}\\ \hline
\multirow{2}{*}{$(MC)$}& without \texttt{CS} &3& 10 & 9 & 8 & 8 \\ 
&\texttt{+CS}& - & 10 & 10 & 10 & 10 \\ \hline
 \multirow{2}{*}{$(PC)$}& without \texttt{CS}&3& 9 & 8 & 8 & 8 \\ 
&\texttt{+CS} & - & 10 & 10 & 10 & 10 \\ \hline
\end{tabular}\caption{Instances N40 solved for $(MC)$ and $(PC)$ within a time limit of 1 hour.}
\label{table:optN40}
\end{table}

We now concentrate on N60 instances. Table \ref{table:n60} compares the performance by  adding cutset inequalities \texttt{CS}, setting an initial feasible solution \texttt{IS} and adding cuts at the root node \texttt{RNC}. We perform this experiment by computing the optimality gap after one hour. Without any of the features mentioned above, the trend on Table \ref{table:perfN102040} is confirmed in $(MC)$ for instances in set N60 where the optimality gap obtained after one hour is smaller in \texttt{AUTO\_BD}, see the first row in Table \ref{table:n60}. However, for $(PC)$ the gap after one hour is slightly better for \texttt{BD\_CW} than for the other methods in this family (see the fifth row in Table \ref{table:n60}). With respect to adding an initial solution, we observe that for $(MC)$ is only profitable for \texttt{BD\_CW+CS}, obtaining in average a $3.5\%$ better optimality gap than without it. The impact of adding an initial solution for $(PC)$ is significant for \texttt{BD\_Trd+CS}, \texttt{BD\_Norm+CS} and \texttt{BD\_CW+CS} obtaining in average solutions with a gap around $4\%$ smaller. However, this improvement is not significant for \texttt{BD\_Auto} for $(PC)$ (see third row of both blocks in Table \ref{table:n60}). Besides, We note that we obtain worse solutions by adding also \texttt{RNC}in both problems with all the algorithms tested. In summary, for the set of instances N60 we have that the best algorithm is \texttt{BD\_CW+CS+IS} for $(MC)$. It decreases the solution gap by around 8\% comparing with the best option of \texttt{Auto\_BD}, which is \texttt{Auto\_BD+CS}. With regard to $(PC)$, the best options are \texttt{BD\_CW+CS+IS} and \texttt{BD\_Norm+CS+IS}, since their solution gaps are around $5.5\%$ smaller than the ones returned by \texttt{Auto\_BD+CS}.

\begin{table}[htpb]
\small
\centering
\begin{tabular}{c c cc cc cc cc}
\cline{3-10}
&& \multicolumn{2}{c}{\texttt{Auto\_BD}} & \multicolumn{2}{c}{\texttt{BD\_Trd}} & \multicolumn{2}{c}{\texttt{BD\_Norm}} & \multicolumn{2}{c}{\texttt{BD\_CW}} \\ \cline{3-10} 
&& \texttt{gap} & \texttt{cuts} & \texttt{gap} & \texttt{cuts} & \texttt{gap} & \texttt{cuts} & \texttt{gap} & \texttt{cuts} \\ \hline
\multirow{4}{*}{$(MC)$}&without\{\texttt{CS}, \texttt{IS}, \texttt{RNC}\}& 38.54 & 6545 & 45.68 & 14068 & 44.53 & 13340 & 43.77 & 16707 \\ 
&\texttt{+CS}& 30.06 & 3729 & 24.27 & 8754 & 22.17 & 8912 & 25.76 & 11378 \\ 
&\texttt{+CS+IS} & 32.90 & 4987 & 27.23 & 9038 & 26.94 & 9469 & 22.27 & 11151 \\
&\texttt{+CS+IS+RNC} &\multicolumn{2}{c}{-}& 37.88 & 8054 & 37.92 & 8230 & 33.58 & 10834 \\ \hline
\multirow{4}{*}{$(PC)$}&without\{\texttt{CS}, \texttt{IS}, \texttt{RNC}\} & 20.49 & 7009 & 20.40 & 14784 & 21.41 & 15501 & 19.93 & 15116 \\ 
&\texttt{+CS} & 15.92 & 5109 & 14.89 & 12354 & 14.09 & 11687 & 14.50 & 11744 \\ 
&\texttt{+CS+IS} & 15.86 & 4372& 11.06 & 8961 & 10.47 & 8490 & 10.44 & 9683 \\ 
&\texttt{+CS+IS+RNC} &\multicolumn{2}{c}{-} & 20.93 & 10971 & 21.28 & 11449 & 19.94 & 11053 \\ \hline
\end{tabular}\caption{Computing gaps to solve N60 for $(MC)$ and $(PC)$ instances comparing the performance of three families of Benders cuts.}
\label{table:n60}
\end{table}

In the following, we analyze the performance of algorithms \texttt{BD\_Norm+CS} \texttt{BD\_CW+CS} when changing parameters $C_{max}$, $\beta$ and $u$ in the corresponding models. In Tables \ref{table:SA MC} and \ref{table:SA PC}, we report average solution times and number of cuts needed to obtain optimal solutions for N40 for different values of these parameters. The instances are grouped by the three different increasing values of the available budget $C_{max}$ (Table \ref{table:SA MC}.a) or $\beta$ (Table \ref{table:SA PC}.a) and private utility $u$ (Tables \ref{table:SA MC}.b and \ref{table:SA PC}.b). For $(MC)$, it is observed that the bigger the value of $C_{max}$ is, the shorter the average solution time is. Table \ref{table:SA MC}.b. shows that the larger the parameter $u$ is, the shorter the solution time for \texttt{BD\_Norm+CS} is. This behavior seems to be different if we are using \texttt{BD\_CW+CS}.

\begin{table}[htpb]
\small
\centering
\begin{minipage}{0.48\textwidth}
\begin{tabular}{c cc cc}
\hline
\multirow{2}{*}{$C_{max}$}& \multicolumn{2}{c}{\texttt{BD\_Norm+CS}} & \multicolumn{2}{c}{\texttt{BD\_CW+CS}} \\ \cline{2-5} 
 & \texttt{t} & \texttt{cuts} & \texttt{t} & \texttt{cuts} \\ \hline
$0.3\,TC$ & 1053.56 & 1580 & 873.58 & 2017 \\ 
$0.5\,TC$ & 622.45 & 2634 & 375.30 & 3358 \\ 
$0.7\,TC$ & 151.24 & 3970 & 177.90 & 5035 \\ \hline
\end{tabular}

\begin{center}a. \end{center}
\end{minipage}
\begin{minipage}{0.48\textwidth}
\centering
\begin{tabular}{c cc cc}
\hline
 \multirow{2}{*}{$u$} & \multicolumn{2}{c}{\texttt{BD\_Norm+CS}} & \multicolumn{2}{c}{\texttt{BD\_CW+CS}} \\ \cline{2-5} 
 & \texttt{t} & \texttt{cuts} & \texttt{t} & \texttt{cuts} \\ \hline
$1.5\,SPath$ & 802.05 & 2792 & 495.84 & 3041 \\ 
$2\,SPath$ & 622.46 & 2634 & 375.30 & 3358 \\ 
$3\,SPath$ & 591.02 & 2674 & 490.28 & 3173 \\ \hline
\end{tabular}

\begin{center}b. \end{center}
\end{minipage}
\caption{Sensitivity analysis for $(MC)$ with $\vert N \vert = 40.$}
\label{table:SA MC}
\end{table}

For $(PC)$, Table \ref{table:SA PC}.a shows that both algorithms take less time to solve the problem to optimality for $\beta = 0.7$ than for $\beta=0.3$ and $\beta=0.5$. \texttt{BD\_CW+CS} is 5 minutes faster in average than \texttt{BD\_Norm+CS} with $\beta = 0.5$. For $\beta = 0.3$ the result is the opposite, \texttt{BD\_Norm+CS} is 100 seconds faster in average than \texttt{BD\_CW+CS}. By varying $u$, we observe that the less the difference between public and private mode distances in the underlying network is, the longer it takes to reach optimality.

\begin{table}[htpb]
\small
\centering
\begin{minipage}{0.48\textwidth}
\centering
\begin{tabular}{c cc cc}
\hline

\multirow{2}{*}{$\beta$}& \multicolumn{2}{c}{\texttt{BD\_Norm+CS}} & \multicolumn{2}{c}{\texttt{BD\_CW+CS}} \\ \cline{2-5} 
 & \texttt{t} & \texttt{cuts} & \texttt{t} & \texttt{cuts} \\ \hline
0.3 & 640.28 & 2675 & 744.95 & 2848 \\ 
0.5 & 697.87 &3673 & 387.40 & 3914 \\ 
0.7 & 273.53 & 3873 & 242.04 & 4460 \\ \hline

\end{tabular}
\begin{center}a. \end{center}
\end{minipage}
\begin{minipage}{0.48\textwidth}
\centering
\begin{tabular}{c cc c c}
\hline
 \multirow{2}{*}{$u$}& \multicolumn{2}{c}{\texttt{BD\_Norm+CS}} & \multicolumn{2}{c}{\texttt{BD\_CW+CS}} \\ \cline{2-5} 
 & \texttt{t} & \texttt{cuts} & \texttt{t} & \texttt{cuts} \\ \hline
 $1.5\,SPath$ & 653.47 & 3625 & 620.79 & 3613 \\
 $2\,SPath$ & 697.87 &3673 & 387.40 & 3914 \\ 
 $3\,SPath$ & 561.43 &3521 & 378.11 & 3643 \\ \hline
\end{tabular}
\begin{center}b. \end{center}
\end{minipage}
\caption{Sensitivity analysis for $(PC)$ with $\vert N \vert = 40.$}
\label{table:SA PC}
\end{table}

\subsection{Performance of algorithms on benchmark instances}

We start by analyzing the Sevilla instance. Tables \ref{table: Sevilla MC} and \ref{table: Sevilla PC} show some results for this instance solved with \texttt{BD\_CW+CS}. Based on this case, figures in Tables \ref{table: Sevilla MC} and \ref{table: Sevilla PC} show the solution graphs for different parameter values. Points not connected in these graphs refer to those nodes that have not been built. The O/D pairs involving some of these nodes are thus not covered. They have been drawn to represent these not covered areas. Data corresponding to each case is collected at the bottom of its figure, in which \texttt{v(ILP)} refers to the objective value. For model $(MC)$, parameter \texttt{cost} represents the cost of the network built, and, for $(PC)$, $G_{cov}$ makes reference to the demand covered. For $(MC)$, we observe that smaller values of $C_{max}$ carry larger solution times as in \textit{random instances}. For $(PC)$, as opposite to {\it random instances}, higher values of $\beta$ are translated in larger solution times. Besides, in this instance, for both models, the shorter the parameter $u$ is, the larger the solution times are.\\
Furthermore, we compare the performance of the \texttt{GRASP} algorithm from \cite{garcia2013grasp} and our implementation \texttt{BD\_CW+CS}. The goal of this experiment is to compare our implementation with a state-of-the-art heuristic for network design problems. We implemented the \texttt{GRASP} algorithm to run $5$ times and return the best solution. Table \ref{table: grasp MC} shows solution times, best value for \texttt{GRASP} (\texttt{Best Value}), the optimality gap, and the optimal value computed with \texttt{BD\_CW+CS}. On the one hand, we observed that the more time \texttt{BD\_CW+CS} takes to compute the optimal solution, the larger the gap of the solution returned by \texttt{GRASP} is. This happens for smaller values of the budget $C_{max}$ and utility $u$. On the other hand, for problems where \texttt{GRASP} obtains small optimality gap, \texttt{BD\_CW+CS} is more efficient to compute the optimal solution. In other words, since \texttt{GRASP} is a constructive algorithm, it is not competitive for instances whose optimal solution captures most of the demand.

\begin{table}[htpb]
\small
\centering
\begin{tabular}{ccccccc}
\hline
 \multirow{2}{*}{$C_{max}$} & \multirow{2}{*}{$u$} & \multicolumn{3}{c}{\texttt{GRASP}} & \multicolumn{2}{c}{\texttt{BD\_CW+CS}} \\ \cline{3-5} \cline{6-7}
 & & \texttt{t} & \texttt{Best Value} & \texttt{gap} & \texttt{t} & \texttt{v(ILP)} \\ \hline
 $0.2\,TC$ & \multirow{3}{*}{$2\,SPath$} & 110.829 & 48629 & 6.97 &1036.11 & 52274 \\ 
 $0.3\,TC$ & & 260.220 & 59828 & 3.96 & 313.07 & 62294 \\
 $0.4\,TC$ & & 396.226 & 63546 & 0.72 & 21.36 & 64011 \\ \hline
\multirow{2}{*}{$0.3\,TC$} & $1.5\,SPath$ & 267.275 & 55778 & 6.97 & 2243.83 & 59958 \\
 & $3\,SPath$ & 225.312 & 62049 & 0.99 & 113.88 & 62670 \\ \hline
\end{tabular}
\caption{Sensitivity analysis for \texttt{GRASP} algorithm \cite{garcia2013grasp} for the Sevilla instance.}
\label{table: grasp MC}
\end{table}

We discuss the results for the Sioux Falls instance, summarized in Tables \ref{table: sioux MC} and \ref{table: sioux PC} in \ref{s:sioux_results}. We observe for $(MC)$, as in the Sevilla network, that the smaller the values of $C_{max}$ and $u$ are, the larger the solution time is. The same is true when varying $\beta$ in $(PC)$, but not for $u$. It takes less time if the difference between both modes of transport is smaller or larger than $2\,SPath$. \\
Our exact method is able to obtain the best quality solution, with a certificate of optimality in reasonable times. Given that network design problems are strategic decisions, having the best quality decision is often more important than the computational times. However, having efficient exact methods as the ones proposed in this article, allows decision makers to perform sensitivity analysis with optimality guarantees in reasonable times.

We also tested our algorithms on benchmark instances \texttt{Germany50} and \texttt{Ta2} form SNDLib (\url{http://sndlib.zib.de/}). We observed that adding cuts at the root node is beneficial for Germany50. We think that this behavior is due to the fact that Germany50 has a denser potential graph (in particular, Germany50 is not a planar graph). The rest of the results obtained for these instances are aligned with the results obtained for Sevilla and Sioux Falls instances. 
For the sake of shortness, this analysis is included in the supplementary material in \url{http://github.com/vbucarey/network_design_coverage/}.

\newcommand{\factor}{0.3}

\begin{table}[htpb]
\small
\centering
\begin{tabular}{|C{0.45\textwidth}|C{0.45\textwidth}|}
\hline
 Underlying Network & $C_{max} = 0.3\,TC$, $u=2\,SPath$ \\
 \includegraphics[width = \factor\textwidth]{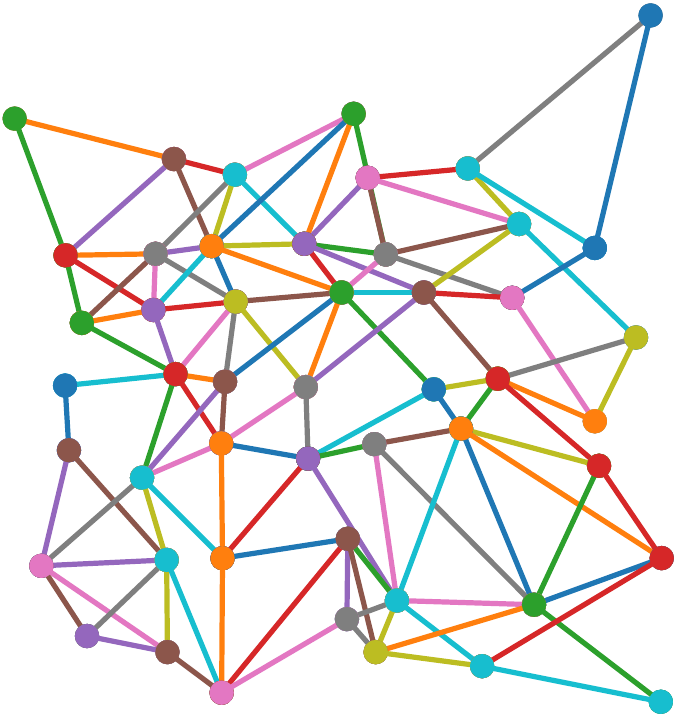} & \includegraphics[width = \factor\textwidth]{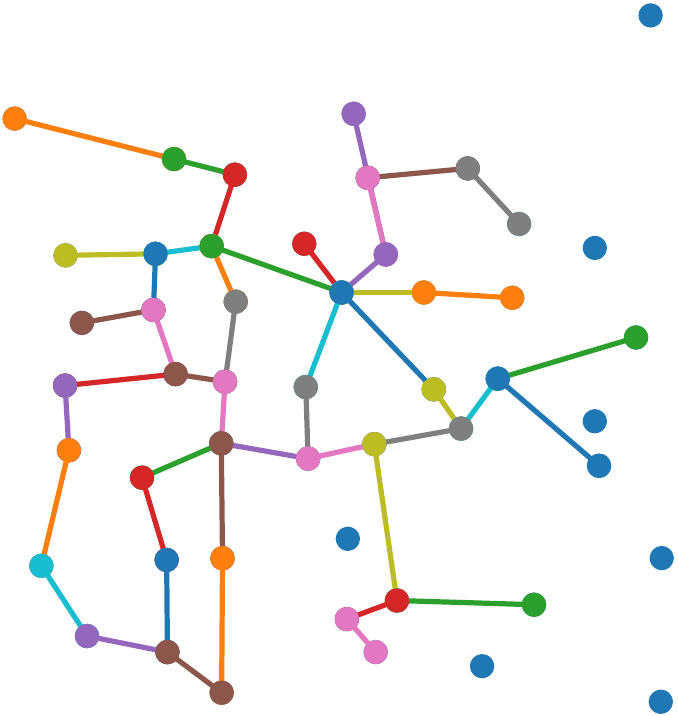} \\
$TC =295297826$, $G =64011$ & \texttt{t} $ =313.07$ , \texttt{cuts} $=7149$ \\ & \texttt{cost} $=88939.76$, \texttt{v(ILP)} $=62294$ \\ \hline 
 $C_{max} = 0.2\,TC$, $u=2\,SPath$ & $C_{max} = 0.4\,TC$, $u=2\,SPath$ \\ 
 \includegraphics[width = \factor\textwidth]{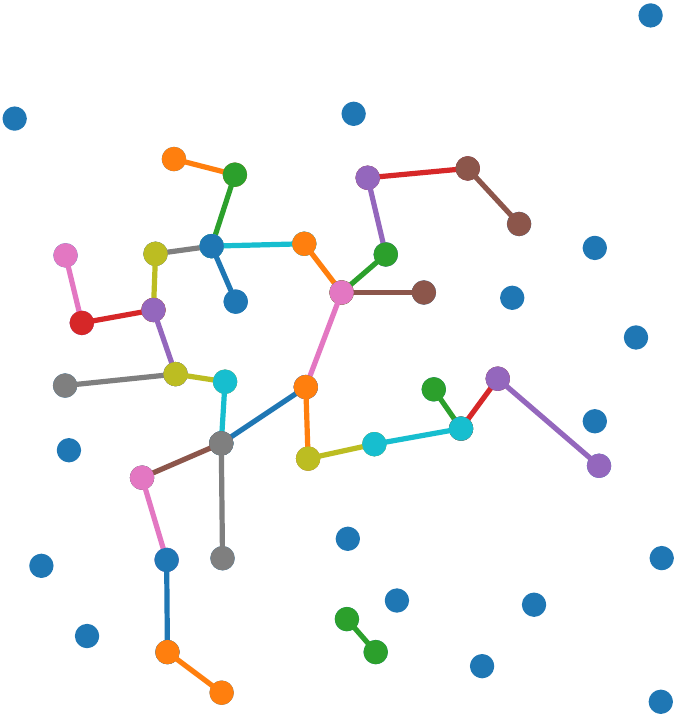} & \includegraphics[width = \factor\textwidth]{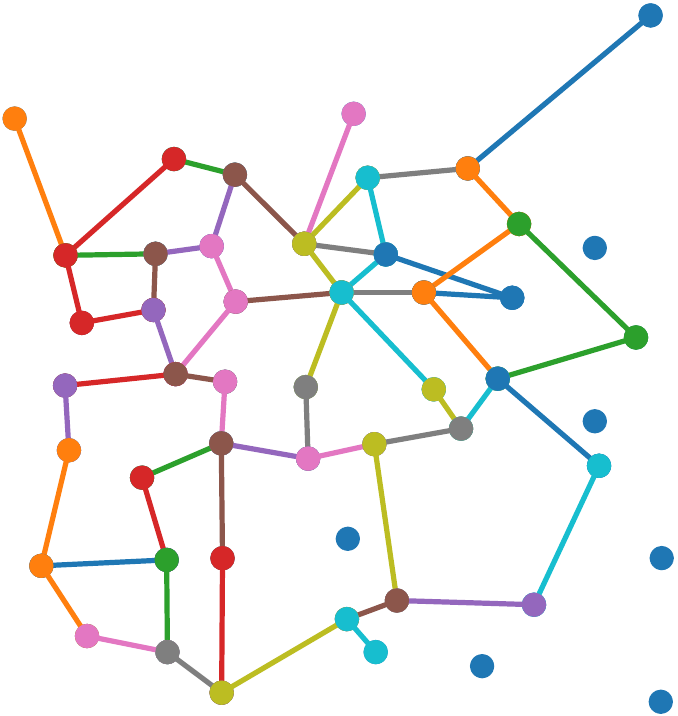} \\ 
 \texttt{t} $ = 1036.11$, \texttt{cuts} $=8428$ & \texttt{t} $ = 21.36$, \texttt{cuts} $=2259$ \\ \texttt{cost} $=53362.53$, v(ILP) $=52274$ & \texttt{cost} $=124353.84$, \texttt{v(ILP)} $=64011$ \\ \hline
 $C_{max} = 0.3\,TC$, $u=1.5\,SPath$ & $C_{max} = 0.3\,TC$, $u=3\,SPath$ \\ 
 \includegraphics[width = \factor\textwidth]{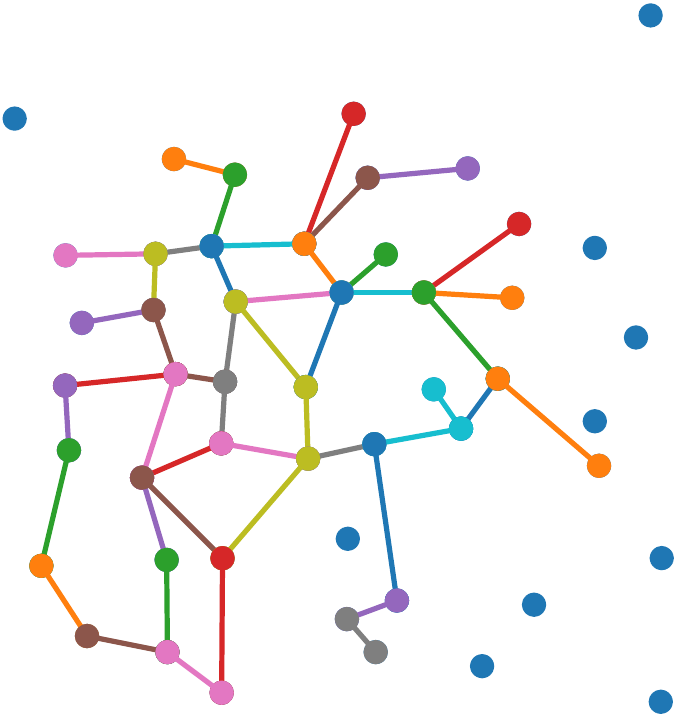} & \includegraphics[width = \factor\textwidth]{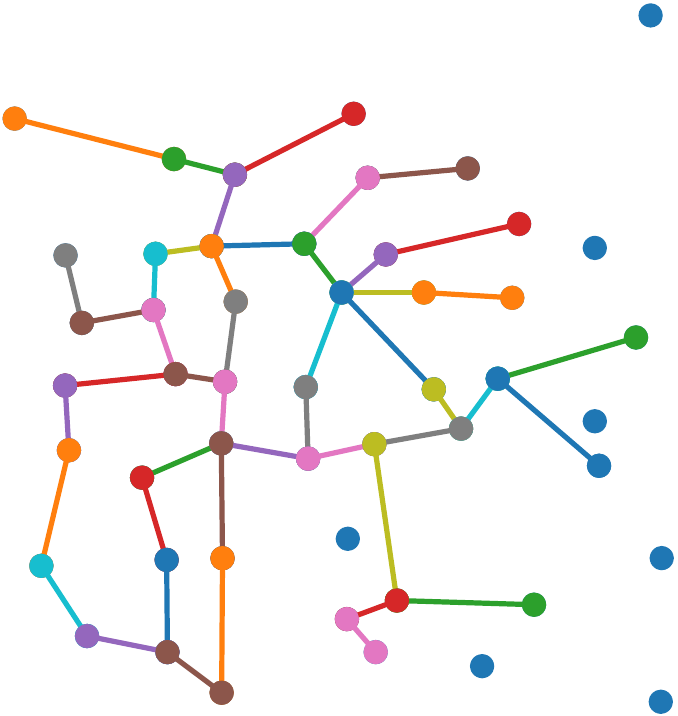} \\ 
 \texttt{t}$ = 2243.83$, \texttt{cuts} $=8537$ & \texttt{t} $ = 113.88$, \texttt{cuts} $=6400$ \\ \texttt{cost} $=88800.91$, \texttt{v(ILP)} $=59958$ & \texttt{cost} $=88879.76$, \texttt{v(ILP)} $=62670$ \\ \hline
\end{tabular}\caption{Sensitivity analysis for the Sevilla Network with $(MC)$.}
\label{table: Sevilla MC}
\end{table}

\begin{table}[htpb]
\small
\centering
\begin{tabular}{|C{0.45\textwidth}|C{0.45\textwidth}|}
\hline
Underlying Network & $\beta = 0.5$, $u=2\,SPath$ \\ 
 \includegraphics[width = \factor\textwidth]{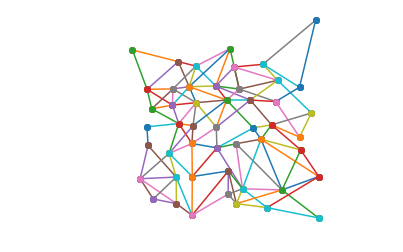} & \includegraphics[width = \factor\textwidth]{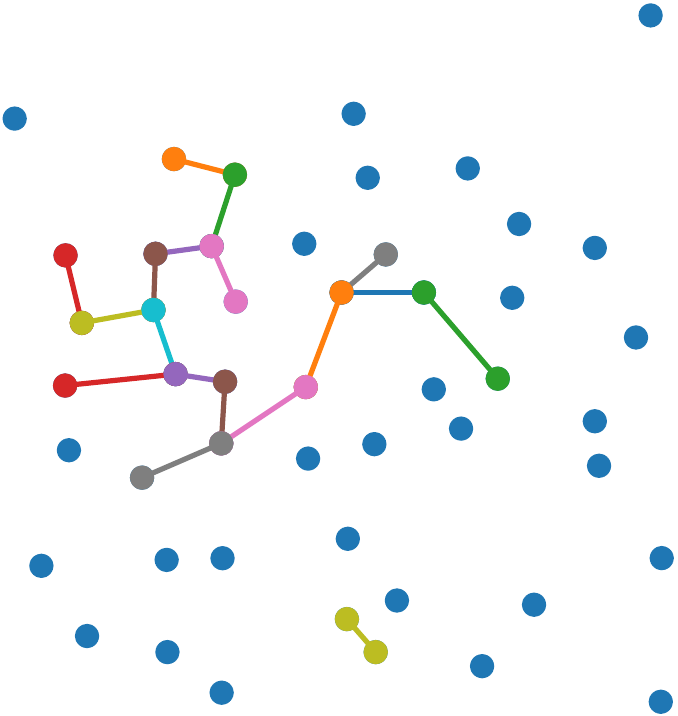} \\ 
 $TC =295297826$, $G =64011$ & \texttt{t} $ = 463.45$, \texttt{cuts} $= 3934$ \\ & $G_{cov} = 32070$, \texttt{v(ILP)} $= 28905.71$ \\ \hline
 $\beta = 0.3$, $u=2\,SPath$ & $\beta = 0.7$, $u=2\,SPath$ \\
 \includegraphics[width = \factor\textwidth]{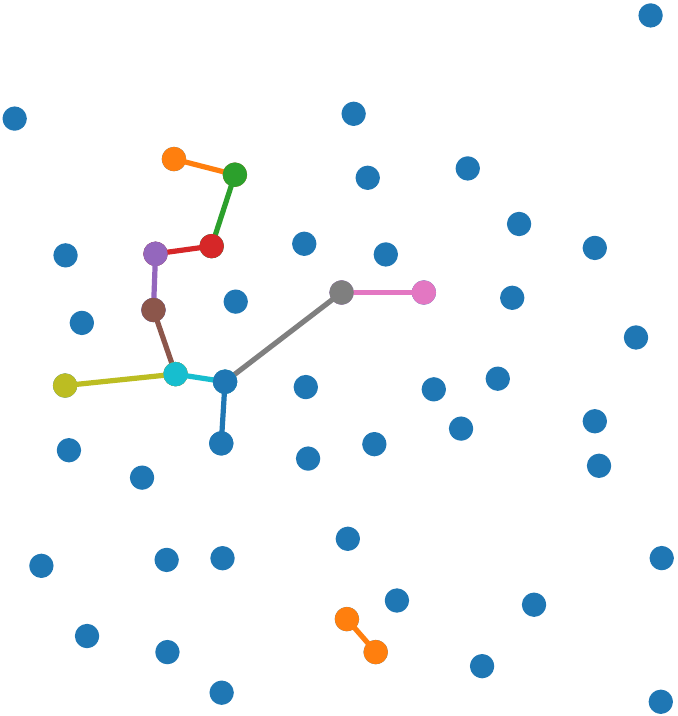} & \includegraphics[width = \factor\textwidth]{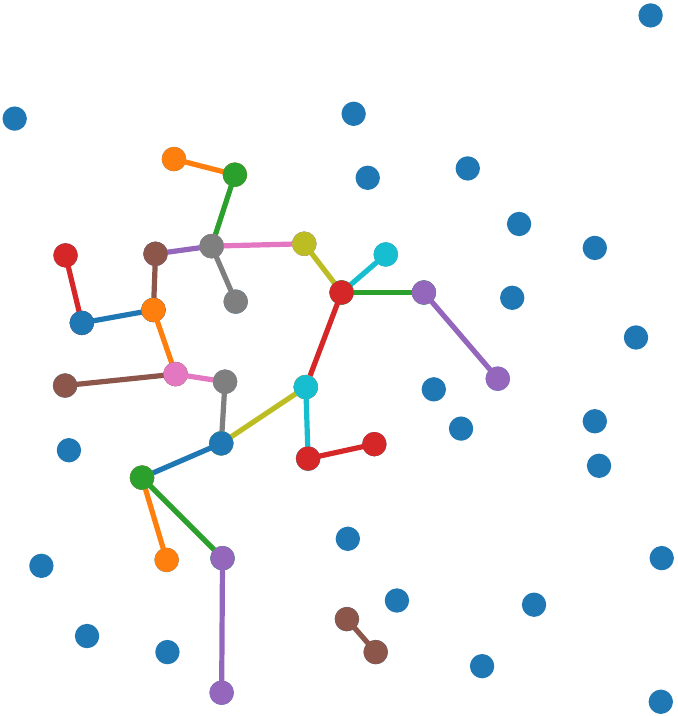} \\ 
 \texttt{t} $ = 353.43$, \texttt{cuts} $= 2294$ & \texttt{t} $ = 532.17$, \texttt{cuts} $= 6070$ \\ $G_{cov} = 19204$, \texttt{v(ILP)} $= 17687.02$ & $G_{cov} = 44830$, \texttt{v(ILP)} $= 42521.65$ \\ \hline
 $\beta = 0.5$, $u=1.5\,SPath$ & $\beta = 0.5$, $u=3\,SPath$ \\
 \includegraphics[width = \factor\textwidth]{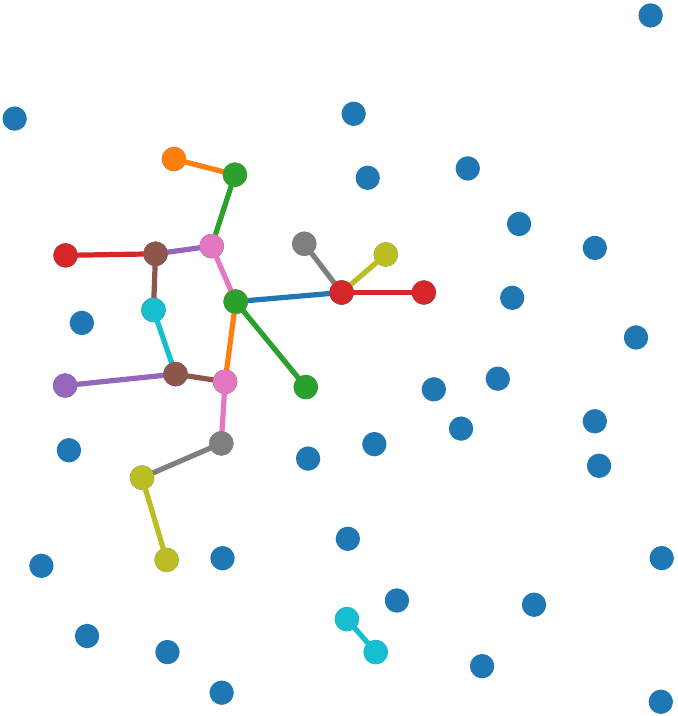} & \includegraphics[width = \factor\textwidth]{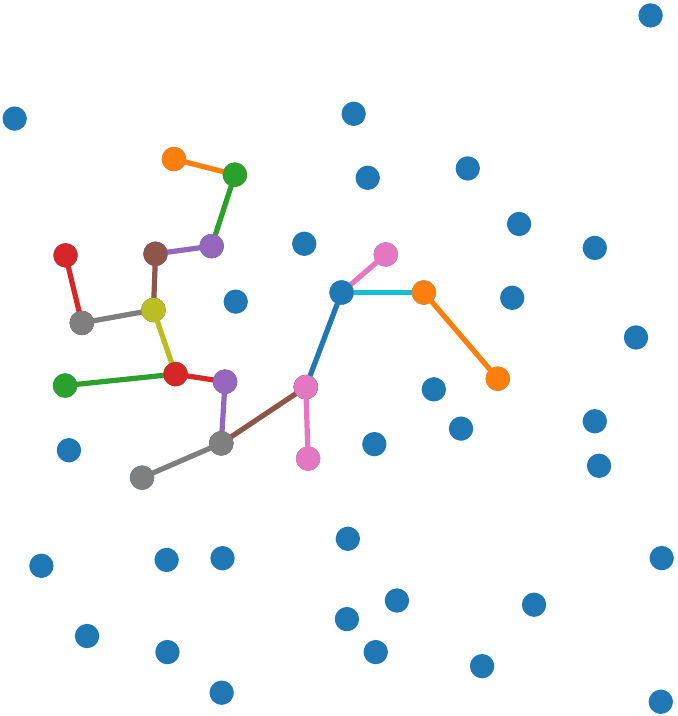} \\ 
 \texttt{t} $ = 1358.20$, \texttt{cuts} $= 4663$ & \texttt{t} $ = 396.56$, \texttt{cuts} $= 3337$ \\ $G_{cov} = 32105$, \texttt{v(ILP)} $= 30562.25$ & $G_{cov} = 32024$, \texttt{v(ILP)} $= 28190.34$ \\ \hline
\end{tabular}\caption{Sensitivity analysis for the Sevilla Network with $(PC)$.}
\label{table: Sevilla PC}
\end{table}

\section{Conclusions} \label{s:conclusion} 

In this paper, we have studied two variants of the \textit{Network Design Problem}: \textit{Maximal Covering Network Design Problem} where we maximize the demand covered under a budget constraint; and \textit{Partial Set Covering Network Design Problem} where the total building cost is minimized subject to a lower bound on the demand covered. We propose mixed integer linear programming formulations that are stronger than existing ones for both problems. We provide some polyhedral properties of these formulations, useful from the algorithmic point of view. We develop exact methods based on Benders decomposition. We also discuss some pre-processing procedures to scale-up the instances solved. These pre-processing techniques play a key role in order to obtain information about the instances and to derive a better algorithmic performance. Our computational results show that the techniques developed in this article allow obtaining better solutions in less time than the techniques in the existing literature. Further research on this topic will focus on the synergy of sophisticated heuristics to find good feasible solutions and decomposition methods, such as the ones presented in this article, to get better bounds and close the optimality gap.
Finally, we remark that objectives of $(MC)$ and $(PC)$ can be included in a bicriteria optimization model. An interesting extension is to exploit the decomposition methods described in this manuscript to the multiobjective setting.

\section*{Acknowledgments}

V\'ictor Bucarey and Martine Labb\'e have been partially supported by the Fonds de la Recherche Scientifique - FNRS under Grant(s) no PDR T0098.18. Natividad González-Blanco and Juan A. Mesa are partially supported by Ministerio de Economía y Competitividad (Spain)/FEDER(UE) under grant MTM2015-67706-P and Ministerio de Ciencia y Tecnolog\'ia(Spain)/FEDER(UE) under grant PID2019-106205GB-I00.

\bibliography{biblio}

\begin{thebibliography}{39}
\expandafter\ifx\csname natexlab\endcsname\relax\def\natexlab#1{#1}\fi
\providecommand{\url}[1]{\texttt{#1}}
\providecommand{\href}[2]{#2}
\providecommand{\path}[1]{#1}
\providecommand{\DOIprefix}{doi:}
\providecommand{\ArXivprefix}{arXiv:}
\providecommand{\URLprefix}{URL: }
\providecommand{\Pubmedprefix}{pmid:}
\providecommand{\doi}[1]{\href{http://dx.doi.org/#1}{\path{#1}}}
\providecommand{\Pubmed}[1]{\href{pmid:#1}{\path{#1}}}
\providecommand{\bibinfo}[2]{#2}
\ifx\xfnm\relax \def\xfnm[#1]{\unskip,\space#1}\fi
\bibitem[{Balakrishnan et~al.(1989)Balakrishnan, Magnanti \&
  Wong}]{balakrishnan1989dual}
\bibinfo{author}{Balakrishnan, A.}, \bibinfo{author}{Magnanti, T.~L.}, \&
  \bibinfo{author}{Wong, R.~T.} (\bibinfo{year}{1989}).
\newblock \bibinfo{title}{A dual-ascent procedure for large-scale uncapacitated
  network design}.
\newblock {\it \bibinfo{journal}{Operations Research}\/},  {\it
  \bibinfo{volume}{37}\/}, \bibinfo{pages}{716--740}.
\bibitem[{Balas \& Perregaard(2002)}]{balas2002lift}
\bibinfo{author}{Balas, E.}, \& \bibinfo{author}{Perregaard, M.}
  (\bibinfo{year}{2002}).
\newblock \bibinfo{title}{Lift-and-project for mixed 0--1 programming: recent
  progress}.
\newblock {\it \bibinfo{journal}{Discrete Applied Mathematics}\/},  {\it
  \bibinfo{volume}{123}\/}, \bibinfo{pages}{129--154}.
\bibitem[{Balas \& Perregaard(2003)}]{balas2003precise}
\bibinfo{author}{Balas, E.}, \& \bibinfo{author}{Perregaard, M.}
  (\bibinfo{year}{2003}).
\newblock \bibinfo{title}{A precise correspondence between lift-and-project
  cuts, simple disjunctive cuts, and mixed integer gomory cuts for 0-1
  programming}.
\newblock {\it \bibinfo{journal}{Mathematical Programming}\/},  {\it
  \bibinfo{volume}{94}\/}, \bibinfo{pages}{221--245}.
\bibitem[{Barahona(1996)}]{Barahona1996network}
\bibinfo{author}{Barahona, F.} (\bibinfo{year}{1996}).
\newblock \bibinfo{title}{Network design using cut inequalities}.
\newblock {\it \bibinfo{journal}{SIAM Journal on optimization}\/},  {\it
  \bibinfo{volume}{6}\/}, \bibinfo{pages}{823--837}.
\bibitem[{Ben-Ameur \& Neto(2007)}]{ben2007acceleration}
\bibinfo{author}{Ben-Ameur, W.}, \& \bibinfo{author}{Neto, J.}
  (\bibinfo{year}{2007}).
\newblock \bibinfo{title}{Acceleration of cutting-plane and column generation
  algorithms: Applications to network design}.
\newblock {\it \bibinfo{journal}{Networks: An International Journal}\/},  {\it
  \bibinfo{volume}{49}\/}, \bibinfo{pages}{3--17}.
\bibitem[{Benders(1962)}]{benders1962partitioning}
\bibinfo{author}{Benders, J.~F.} (\bibinfo{year}{1962}).
\newblock \bibinfo{title}{Partitioning procedures for solving mixed-variables
  programming problems}.
\newblock {\it \bibinfo{journal}{Numerische Mathematik}\/},  {\it
  \bibinfo{volume}{4}\/}, \bibinfo{pages}{238--252}.
\bibitem[{Berge(1957)}]{berge1957two}
\bibinfo{author}{Berge, C.} (\bibinfo{year}{1957}).
\newblock \bibinfo{title}{Two theorems in graph theory}.
\newblock {\it \bibinfo{journal}{Proceedings of the National Academy of
  Sciences of the United States of America}\/},  {\it \bibinfo{volume}{43}\/},
  \bibinfo{pages}{842--844}.
\bibitem[{Botton et~al.(2013)Botton, Fortz, Gouveia \&
  Poss}]{botton2013benders}
\bibinfo{author}{Botton, Q.}, \bibinfo{author}{Fortz, B.},
  \bibinfo{author}{Gouveia, L.}, \& \bibinfo{author}{Poss, M.}
  (\bibinfo{year}{2013}).
\newblock \bibinfo{title}{Benders decomposition for the hop-constrained
  survivable network design problem}.
\newblock {\it \bibinfo{journal}{INFORMS journal on computing}\/},  {\it
  \bibinfo{volume}{25}\/}, \bibinfo{pages}{13--26}.
\bibitem[{Canca et~al.(2017)Canca, De-Los-Santos, Laporte \&
  Mesa}]{canca2017adaptive}
\bibinfo{author}{Canca, D.}, \bibinfo{author}{De-Los-Santos, A.},
  \bibinfo{author}{Laporte, G.}, \& \bibinfo{author}{Mesa, J.~A.}
  (\bibinfo{year}{2017}).
\newblock \bibinfo{title}{An adaptive neighborhood search metaheuristic for the
  integrated railway rapid transit network design and line planning problem}.
\newblock {\it \bibinfo{journal}{Computers \& Operations Research}\/},  {\it
  \bibinfo{volume}{78}\/}, \bibinfo{pages}{1--14}.
\bibitem[{Canca et~al.(2019)Canca, De-Los-Santos, Laporte \&
  Mesa}]{canca2019integrated}
\bibinfo{author}{Canca, D.}, \bibinfo{author}{De-Los-Santos, A.},
  \bibinfo{author}{Laporte, G.}, \& \bibinfo{author}{Mesa, J.~A.}
  (\bibinfo{year}{2019}).
\newblock \bibinfo{title}{Integrated railway rapid transit network design and
  line planning problem with maximum profit}.
\newblock {\it \bibinfo{journal}{Transportation Research Part E: Logistics and
  Transportation Review}\/},  {\it \bibinfo{volume}{127}\/},
  \bibinfo{pages}{1--30}.
\bibitem[{Cascetta(2009)}]{cascetta2009transportation}
\bibinfo{author}{Cascetta, E.} (\bibinfo{year}{2009}).
\newblock {\it \bibinfo{title}{Transportation systems analysis: models and
  applications}\/} volume~\bibinfo{volume}{29}.
\newblock \bibinfo{publisher}{Springer Science \& Business Media}.
\bibitem[{Church \& ReVelle(1974{\natexlab{a}})}]{Church1974}
\bibinfo{author}{Church, R.}, \& \bibinfo{author}{ReVelle, C.}
  (\bibinfo{year}{1974}{\natexlab{a}}).
\newblock \bibinfo{title}{The maximal covering location problem}.
\newblock {\it \bibinfo{journal}{Papers of the \uppercase{R}egional
  \uppercase{S}cience \uppercase{A}ssociation}\/},  {\it
  \bibinfo{volume}{32}\/}, \bibinfo{pages}{101--118}.
\bibitem[{Church \& ReVelle(1974{\natexlab{b}})}]{church1974maximal}
\bibinfo{author}{Church, R.}, \& \bibinfo{author}{ReVelle, C.}
  (\bibinfo{year}{1974}{\natexlab{b}}).
\newblock \bibinfo{title}{The maximal covering location problem}.
\newblock In {\it \bibinfo{booktitle}{Papers of the regional science
  association}\/} (pp. \bibinfo{pages}{101--118}).
\newblock \bibinfo{organization}{Springer-Verlag} volume~\bibinfo{volume}{32}.
\bibitem[{Conforti \& Wolsey(2019)}]{Conforti2019facet}
\bibinfo{author}{Conforti, M.}, \& \bibinfo{author}{Wolsey, L.~A.}
  (\bibinfo{year}{2019}).
\newblock \bibinfo{title}{“{F}acet” separation with one linear program}.
\newblock {\it \bibinfo{journal}{Mathematical Programming}\/},  {\it
  \bibinfo{volume}{178}\/}, \bibinfo{pages}{361--380}.
\bibitem[{Cordeau et~al.(2019)Cordeau, Furini \&
  Ljubi{\'c}}]{Cordeau2019benders}
\bibinfo{author}{Cordeau, J.-F.}, \bibinfo{author}{Furini, F.}, \&
  \bibinfo{author}{Ljubi{\'c}, I.} (\bibinfo{year}{2019}).
\newblock \bibinfo{title}{Benders decomposition for very large scale partial
  set covering and maximal covering location problems}.
\newblock {\it \bibinfo{journal}{European Journal of Operational Research}\/},
  {\it \bibinfo{volume}{275}\/}, \bibinfo{pages}{882--896}.
\bibitem[{Costa et~al.(2009)Costa, Cordeau \& Gendron}]{Costa2009benders}
\bibinfo{author}{Costa, A.~M.}, \bibinfo{author}{Cordeau, J.-F.}, \&
  \bibinfo{author}{Gendron, B.} (\bibinfo{year}{2009}).
\newblock \bibinfo{title}{Benders, metric and cutset inequalities for
  multicommodity capacitated network design}.
\newblock {\it \bibinfo{journal}{Computational Optimization and
  Applications}\/},  {\it \bibinfo{volume}{42}\/}, \bibinfo{pages}{371--392}.
\bibitem[{Desrochers(1986)}]{Desrochers1986algorithm}
\bibinfo{author}{Desrochers, M.} (\bibinfo{year}{1986}).
\newblock {\it \bibinfo{title}{An algorithm for the shortest path problem with
  resource constraints}\/} volume \bibinfo{volume}{421}.
\newblock \bibinfo{publisher}{Universit{\'e} de Montr{\'e}al, Centre de
  recherche sur les transports}.
\bibitem[{Fischetti et~al.(2010)Fischetti, Salvagnin \&
  Zanette}]{Fischetti2010note}
\bibinfo{author}{Fischetti, M.}, \bibinfo{author}{Salvagnin, D.}, \&
  \bibinfo{author}{Zanette, A.} (\bibinfo{year}{2010}).
\newblock \bibinfo{title}{A note on the selection of \uppercase{B}enders’
  cuts}.
\newblock {\it \bibinfo{journal}{Mathematical Programming}\/},  {\it
  \bibinfo{volume}{124}\/}, \bibinfo{pages}{175--182}.
\bibitem[{Fortz et~al.(2021)Fortz, Gouveia \& Moura}]{fortz2020steiner}
\bibinfo{author}{Fortz, B.}, \bibinfo{author}{Gouveia, L.}, \&
  \bibinfo{author}{Moura, P.} (\bibinfo{year}{2021}).
\newblock {\it \bibinfo{title}{A comparison of node-based and arc-based
  hop-indexed formulations for the Steiner tree problem with hop
  constraints}\/}.
\newblock \bibinfo{type}{Technical Report} Université libre de Bruxelles.
\bibitem[{Fortz \& Poss(2009)}]{FORTZ2009359}
\bibinfo{author}{Fortz, B.}, \& \bibinfo{author}{Poss, M.}
  (\bibinfo{year}{2009}).
\newblock \bibinfo{title}{An improved benders decomposition applied to a
  multi-layer network design problem}.
\newblock {\it \bibinfo{journal}{Operations Research Letters}\/},  {\it
  \bibinfo{volume}{37}\/}, \bibinfo{pages}{359 -- 364}.
\bibitem[{Garc\'ia \& Mar\'in(2020)}]{GarciaMarin2020}
\bibinfo{author}{Garc\'ia, S.}, \& \bibinfo{author}{Mar\'in, A.}
  (\bibinfo{year}{2020}).
\newblock \bibinfo{title}{Covering location problems}.
\newblock In \bibinfo{editor}{G.~Laporte}, \bibinfo{editor}{N.~Stefan}, \&
  \bibinfo{editor}{F.~S. da~Gama} (Eds.), {\it \bibinfo{booktitle}{Location
  Science}\/} (pp. \bibinfo{pages}{99--119}).
\newblock \bibinfo{publisher}{Springer}.
\bibitem[{Garc{\'\i}a-Archilla et~al.(2013)Garc{\'\i}a-Archilla, Lozano, Mesa
  \& Perea}]{garcia2013grasp}
\bibinfo{author}{Garc{\'\i}a-Archilla, B.}, \bibinfo{author}{Lozano, A.~J.},
  \bibinfo{author}{Mesa, J.~A.}, \& \bibinfo{author}{Perea, F.}
  (\bibinfo{year}{2013}).
\newblock \bibinfo{title}{Grasp algorithms for the robust railway network
  design problem}.
\newblock {\it \bibinfo{journal}{Journal of Heuristics}\/},  {\it
  \bibinfo{volume}{19}\/}, \bibinfo{pages}{399--422}.
\bibitem[{Guihaire \& Hao(2008)}]{guihaire2008transit}
\bibinfo{author}{Guihaire, V.}, \& \bibinfo{author}{Hao, J.-K.}
  (\bibinfo{year}{2008}).
\newblock \bibinfo{title}{Transit network design and scheduling: A global
  review}.
\newblock {\it \bibinfo{journal}{Transportation Research Part A: Policy and
  Practice}\/},  {\it \bibinfo{volume}{42}\/}, \bibinfo{pages}{1251--1273}.
\bibitem[{Hakimi(1965)}]{hakimi1965optimum}
\bibinfo{author}{Hakimi, S.~L.} (\bibinfo{year}{1965}).
\newblock \bibinfo{title}{Optimum distribution of switching centers in a
  communication network and some related graph theoretic problems}.
\newblock {\it \bibinfo{journal}{Operations research}\/},  {\it
  \bibinfo{volume}{13}\/}, \bibinfo{pages}{462--475}.
\bibitem[{Hellman(2013)}]{sioux}
\bibinfo{author}{Hellman, F.} (\bibinfo{year}{2013}).
\newblock {\it \bibinfo{title}{Sioux Falls Variants for Network Design}\/}.
\newblock \URLprefix
  \url{http://www.bgu.ac.il/~bargera/tntp/SiouxFalls_CNDP/SiouxFallsVariantsForNetworkDesign.html}
  \bibinfo{note}{accessed April 24th, 2021}.
\bibitem[{Koster et~al.(2013)Koster, Phan \& Tieves}]{Koster2013extended}
\bibinfo{author}{Koster, A.}, \bibinfo{author}{Phan, T.~K.}, \&
  \bibinfo{author}{Tieves, M.} (\bibinfo{year}{2013}).
\newblock \bibinfo{title}{Extended cutset inequalities for the network power
  consumption problem}.
\newblock {\it \bibinfo{journal}{Electronic Notes in Discrete Mathematics}\/},
  {\it \bibinfo{volume}{41}\/}, \bibinfo{pages}{69--76}.
\bibitem[{Kr{\'o}l \& Kr{\'o}l(2019)}]{krol2019design}
\bibinfo{author}{Kr{\'o}l, A.}, \& \bibinfo{author}{Kr{\'o}l, M.}
  (\bibinfo{year}{2019}).
\newblock \bibinfo{title}{The design of a metro network using a genetic
  algorithm}.
\newblock {\it \bibinfo{journal}{Applied Sciences}\/},  {\it
  \bibinfo{volume}{9}\/}, \bibinfo{pages}{433}.
\bibitem[{Ljubi{\'c} et~al.(2019)Ljubi{\'c}, Mouaci, Perrot \&
  Gourdin}]{ljubic2020benders}
\bibinfo{author}{Ljubi{\'c}, I.}, \bibinfo{author}{Mouaci, A.},
  \bibinfo{author}{Perrot, N.}, \& \bibinfo{author}{Gourdin, {\'E}.}
  (\bibinfo{year}{2019}).
\newblock \bibinfo{title}{Benders decomposition for a node-capacitated virtual
  network functions placement and routing problem}.
\bibitem[{Ljubi{\'c} et~al.(2012)Ljubi{\'c}, Putz \&
  Salazar-Gonz{\'a}lez}]{Ljubic2012exact}
\bibinfo{author}{Ljubi{\'c}, I.}, \bibinfo{author}{Putz, P.}, \&
  \bibinfo{author}{Salazar-Gonz{\'a}lez, J.-J.} (\bibinfo{year}{2012}).
\newblock \bibinfo{title}{Exact approaches to the single-source network loading
  problem}.
\newblock {\it \bibinfo{journal}{Networks}\/},  {\it \bibinfo{volume}{59}\/},
  \bibinfo{pages}{89--106}.
\bibitem[{Magnanti et~al.(1986)Magnanti, Mireault \&
  Wong}]{magnanti1986tailoring}
\bibinfo{author}{Magnanti, T.~L.}, \bibinfo{author}{Mireault, P.}, \&
  \bibinfo{author}{Wong, R.~T.} (\bibinfo{year}{1986}).
\newblock \bibinfo{title}{Tailoring \uppercase{B}enders decomposition for
  uncapacitated network design}.
\newblock In {\it \bibinfo{booktitle}{Netflow at Pisa}\/} (pp.
  \bibinfo{pages}{112--154}).
\newblock \bibinfo{publisher}{Springer}.
\bibitem[{Magnanti \& Wong(1981)}]{Magnanti1981accelerating}
\bibinfo{author}{Magnanti, T.~L.}, \& \bibinfo{author}{Wong, R.~T.}
  (\bibinfo{year}{1981}).
\newblock \bibinfo{title}{Accelerating \uppercase{B}enders decomposition:
  Algorithmic enhancement and model selection criteria}.
\newblock {\it \bibinfo{journal}{Operations research}\/},  {\it
  \bibinfo{volume}{29}\/}, \bibinfo{pages}{464--484}.
\bibitem[{Magnanti \& Wong(1984)}]{MagnantiWong1984}
\bibinfo{author}{Magnanti, T.~L.}, \& \bibinfo{author}{Wong, R.~T.}
  (\bibinfo{year}{1984}).
\newblock \bibinfo{title}{Network design and transportation planning: Models
  and algorithms}.
\newblock {\it \bibinfo{journal}{Transportation Science}\/},  (pp.
  \bibinfo{pages}{1--55}).
\bibitem[{Mar{\'\i}n \& Jaramillo(2009)}]{marin2009urban}
\bibinfo{author}{Mar{\'\i}n, {\'A}.~G.}, \& \bibinfo{author}{Jaramillo, P.}
  (\bibinfo{year}{2009}).
\newblock \bibinfo{title}{Urban rapid transit network design: accelerated
  \uppercase{B}enders decomposition}.
\newblock {\it \bibinfo{journal}{Annals of Operations Research}\/},  {\it
  \bibinfo{volume}{169}\/}, \bibinfo{pages}{35--53}.
\bibitem[{Norman \& Rabin(1959)}]{norman1959algorithm}
\bibinfo{author}{Norman, R.~Z.}, \& \bibinfo{author}{Rabin, M.~O.}
  (\bibinfo{year}{1959}).
\newblock \bibinfo{title}{An algorithm for a minimum cover of a graph}.
\newblock {\it \bibinfo{journal}{Proceedings of the American Mathematical
  Society}\/},  {\it \bibinfo{volume}{10}\/}, \bibinfo{pages}{315--319}.
\bibitem[{Perea et~al.(2020)Perea, Menezes, Mesa \&
  Rubio-Del-Rey}]{perea2020transportation}
\bibinfo{author}{Perea, F.}, \bibinfo{author}{Menezes, M.~B.},
  \bibinfo{author}{Mesa, J.~A.}, \& \bibinfo{author}{Rubio-Del-Rey, F.}
  (\bibinfo{year}{2020}).
\newblock \bibinfo{title}{Transportation infrastructure network design in the
  presence of modal competition: computational complexity classification and a
  genetic algorithm}.
\newblock {\it \bibinfo{journal}{TOP}\/},  {\it \bibinfo{volume}{28}\/},
  \bibinfo{pages}{442--474}.
\bibitem[{Rahmaniani et~al.(2017)Rahmaniani, Crainic, Gendreau \&
  Rei}]{rahmaniani2017benders}
\bibinfo{author}{Rahmaniani, R.}, \bibinfo{author}{Crainic, T.~G.},
  \bibinfo{author}{Gendreau, M.}, \& \bibinfo{author}{Rei, W.}
  (\bibinfo{year}{2017}).
\newblock \bibinfo{title}{The \uppercase{B}enders decomposition algorithm: A
  literature review}.
\newblock {\it \bibinfo{journal}{European Journal of Operational Research}\/},
  {\it \bibinfo{volume}{259}\/}, \bibinfo{pages}{801--817}.
\bibitem[{Schmidt \& Sch{\"o}bel(2014)}]{schmidt2014location}
\bibinfo{author}{Schmidt, M.}, \& \bibinfo{author}{Sch{\"o}bel, A.}
  (\bibinfo{year}{2014}).
\newblock \bibinfo{title}{Location of speed-up subnetworks}.
\newblock {\it \bibinfo{journal}{Annals of Operations Research}\/},  {\it
  \bibinfo{volume}{223}\/}, \bibinfo{pages}{379--401}.
\bibitem[{Sinnl \& Ljubi{\'c}(2016)}]{sinnl2016node}
\bibinfo{author}{Sinnl, M.}, \& \bibinfo{author}{Ljubi{\'c}, I.}
  (\bibinfo{year}{2016}).
\newblock \bibinfo{title}{A node-based layered graph approach for the steiner
  tree problem with revenues, budget and hop-constraints}.
\newblock {\it \bibinfo{journal}{Mathematical Programming Computation}\/},
  {\it \bibinfo{volume}{8}\/}, \bibinfo{pages}{461--490}.
\bibitem[{Toregas et~al.(1971)Toregas, Swain, ReVelle \&
  Bergman}]{toregas1971location}
\bibinfo{author}{Toregas, C.}, \bibinfo{author}{Swain, R.},
  \bibinfo{author}{ReVelle, C.}, \& \bibinfo{author}{Bergman, L.}
  (\bibinfo{year}{1971}).
\newblock \bibinfo{title}{The location of emergency service facilities}.
\newblock {\it \bibinfo{journal}{Operations research}\/},  {\it
  \bibinfo{volume}{19}\/}, \bibinfo{pages}{1363--1373}.

\end{thebibliography}

\newpage
\begin{appendix}

\section{Pseudo-code for initial feasible solutions} \label{ap:mipstart}

In this section we provide the pseudo-codes to determine an initial feasible solution for $(MC)$ and $(PC)$ described in Section \ref{s:initial_solution}. We denote by $N_s, E_s$ and $W_s$ the set of indices of design and mode choice variables set to 1 at the end of each algorithm.

\begin{algorithm}[H] 
	\caption{Initial Feasible Solution for $(MC)$} 
	\label{init_sol_MC} 
	\begin{algorithmic} 
	\STATE {\bf Initialization:} Set $N_s = \emptyset$, $E_s = \emptyset$ and $ W_s = \emptyset$ and $IC$ = 0. $\quad$
	\STATE Compute ratio $r_w = \frac{g^w}{C(\mbox{Path}_w)}$:
	
	\FOR{$w \in W$ in decreasing order of $r_w$}
		\STATE $\bar C = C(\mbox{Path}_w) - \sum_{e \in E_s \cap\widetilde{E}^w}c_e - \sum_{i \in N_s \cap\widetilde{N}^w} b_i$
		\IF{$IC + \bar C \le C_{max}$}
		 \STATE $W_s \leftarrow W_s \cup \{w\}$
		 \STATE $E_s \leftarrow E_s \cup \widetilde{E}^w$
		 \STATE $N_s \leftarrow N_s \cup \widetilde{N}^w$
		 \STATE $IC \leftarrow IC + \bar C$
		\ENDIF
	\ENDFOR
	\STATE $x_e = 1$ for $e \in E_s$, 0 otherwise.
	\STATE $y_i = 1$ for $i \in N_s$, 0 otherwise.
	\STATE $z^w = 1$ for $w \in W_s$, 0 otherwise.
	\RETURN $(x,y,z)$
	\end{algorithmic}

\end{algorithm}

\begin{algorithm}[H] 
	\caption{Initial Feasible Solution for $(PC)$} 
	\label{init_sol_PC} 
	\begin{algorithmic} 
	\STATE {\bf Initialization:} Set $\bar W_s = W$ and $Z_s = Z_{total}$. $\quad$
	\STATE Compute ratio $r_w = \frac{g^w}{C(\mbox{Path}_w)}$:
	
	\FOR{$w \in W$ in decreasing order of $r_w$}
		\IF{$Z_s - g^w \ge \beta\, Z_{total}$}
		 \STATE $W_s \leftarrow W_s \setminus \{w\}$
		 \STATE $Z_s \leftarrow Z_s - g^w$
		\ENDIF
	\ENDFOR
	\STATE $x_e = 1$ if $e \in \bigcup_{w \in W_s} \widetilde{E}^w$, 0 otherwise.
	\STATE $y_i = 1$ if $i \in \bigcup_{w \in W_s} \widetilde{N}^w$, 0 otherwise.
	\STATE $z^w = 1$ for $w \in W_s$, 0 otherwise.
	\RETURN $(x,y,z)$
	\end{algorithmic}
\end{algorithm}

\section{Results for SIOUX Falls networks} \label{s:sioux_results}
\renewcommand{\factor}{0.23}
\begin{table}[H]
\small
\centering
\begin{tabular}{|C{0.45\textwidth}|C{0.45\textwidth}|}
\hline
 Underlying Network & $C_{max} = 0.5\,TC$, $u=2\,SPath$ \\ \includegraphics[width = \factor\textwidth]{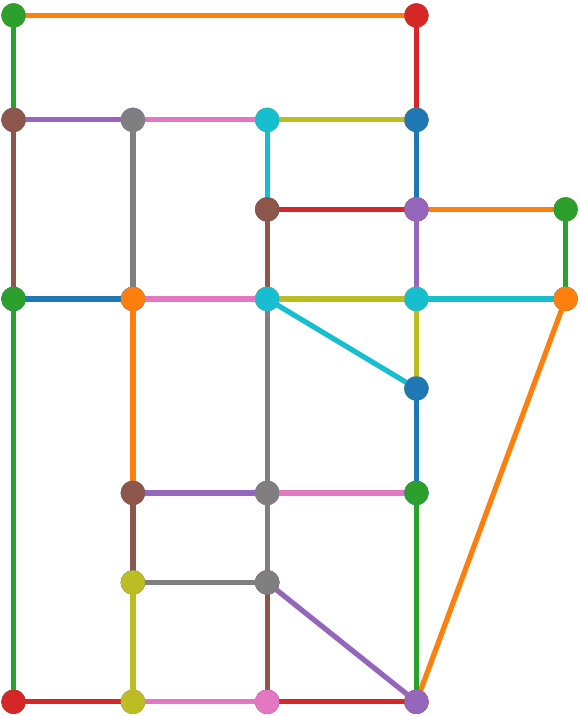} & \includegraphics[width = \factor\textwidth]{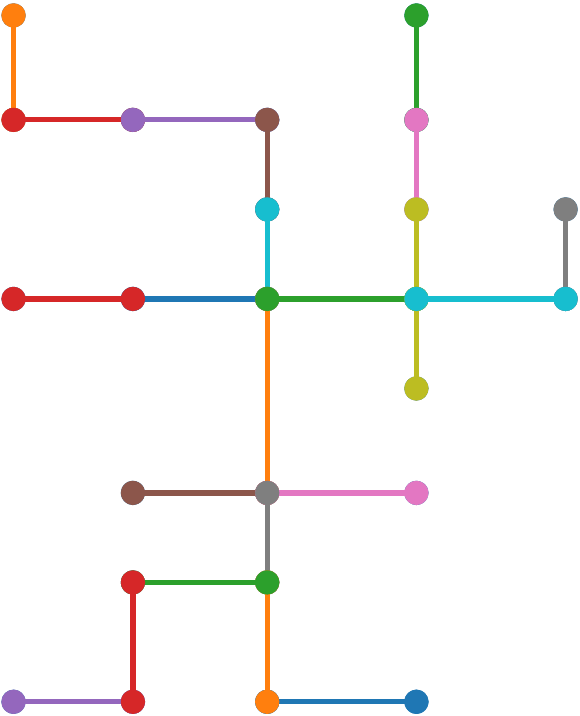} \\ 
 $TC = 4171$, $G = 84437$ & \texttt{t} $ = 22.85$, \texttt{cuts} $ = 3496$ \\ & \texttt{cost} $ = 2070$, \texttt{v(ILP)} $ = 75488$ \\ \hline
 $C_{max} = 0.3\,TC$, $u=2\,SPath$ & $C_{max} = 0.7\,TC$, $u=2\,SPath$ \\ 
 \includegraphics[width = \factor\textwidth]{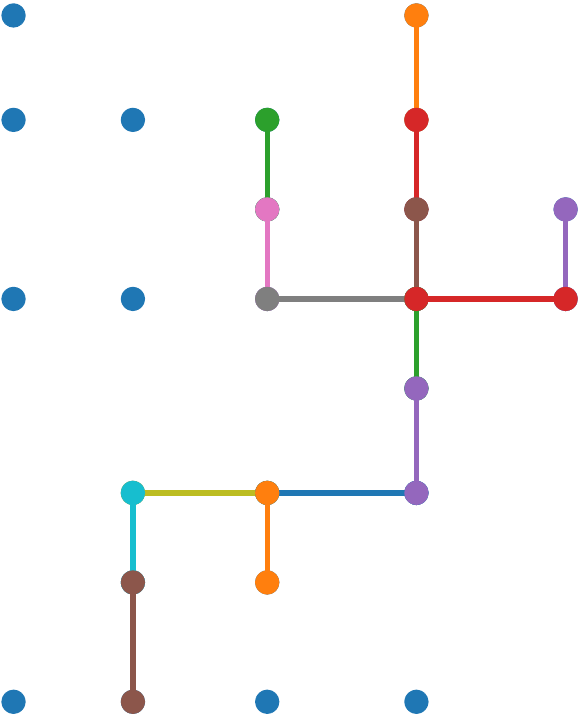} & \includegraphics[width = \factor\textwidth]{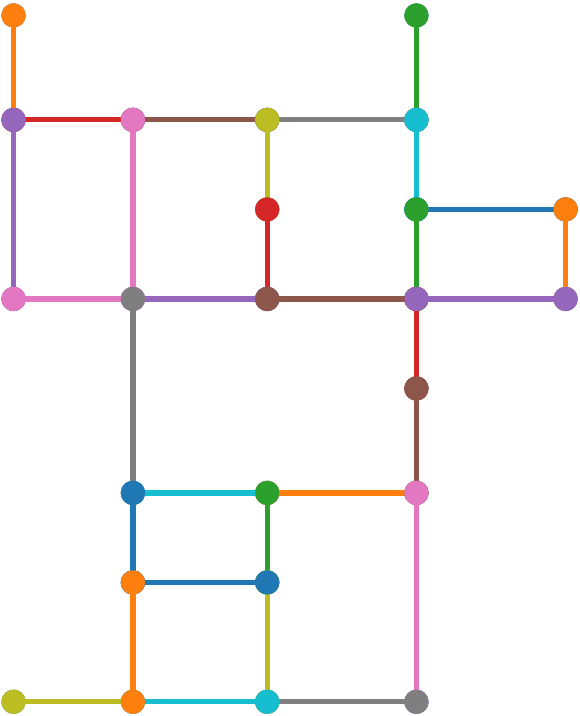} \\ 
 \texttt{t}$ = 458.84$, \texttt{cuts} $ = 3056$ & \texttt{t} $ = 2.73$, \texttt{cuts} $ = 801$ \\ \texttt{cost} $ = 1237$, \texttt{v(ILP)} $ = 35039$ & \texttt{cost} $ = 2870$, \texttt{v(ILP)} $ = 82699$ \\ \hline
 $C_{max} = 0.5\,TC$, $u=1.5\, SPath$ & $C_{max} = 0.5\,TC$, $u=3\,SPath$ \\ 
 \includegraphics[width = \factor\textwidth]{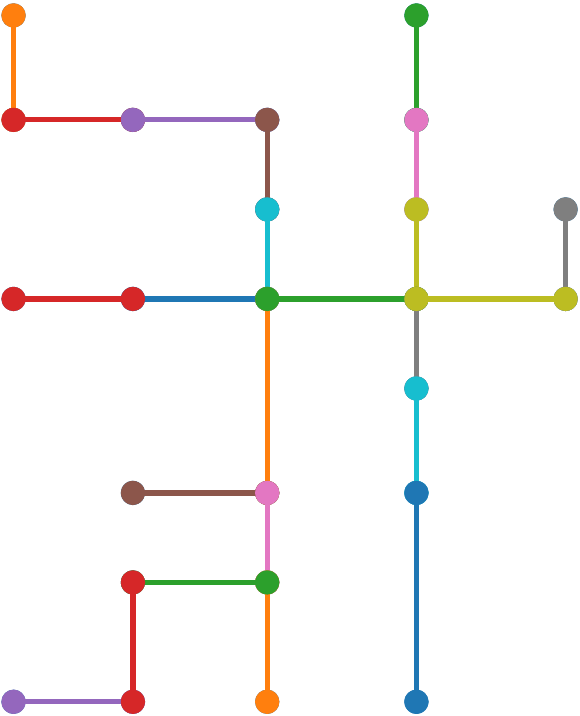} & \includegraphics[width = \factor\textwidth]{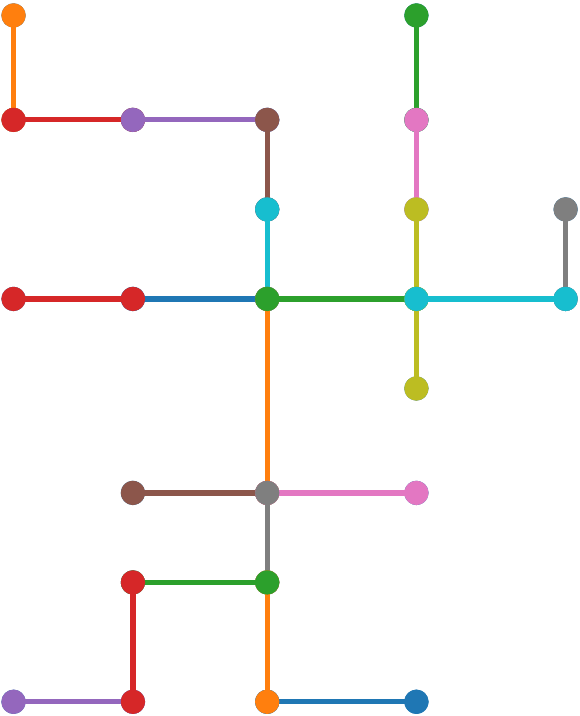} \\ 
 \texttt{t}$ = 60.31$, \texttt{cuts} $ = 3460$ & \texttt{t} $ = 14.17$, \texttt{cuts} $ = 2641$ \\ \texttt{cost} $ = 2080$, \texttt{v(ILP)} $ = 68227$ & \texttt{cost} $ = 2070$, \texttt{v(ILP)} $ = 75488$ \\ \hline
\end{tabular}\caption{Sensitivity analysis for the Sioux Falls Network with $(MC)$.}
\label{table: sioux MC}
\end{table}

\begin{table}[htpb]
\small
\centering
\begin{tabular}{|C{0.45\textwidth}|C{0.45\textwidth}|}
\hline
 Underlying Network & $\beta = 0.5$, $u=2\,SPath$ \\ 
 \includegraphics[width = \factor\textwidth]{figs/potentialNetworkSIOUX} & \includegraphics[width = \factor\textwidth]{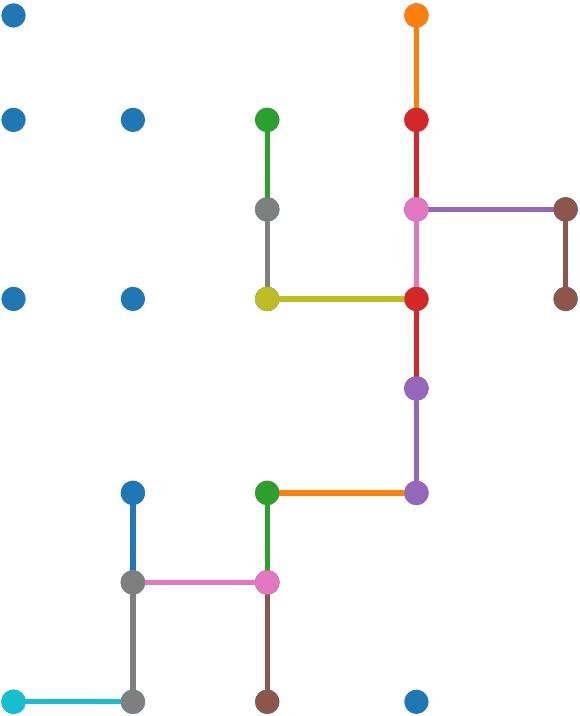} \\
 $TC = 4171$, $G = 84437$ & \texttt{t} $ = 429.85$, \texttt{cuts} $ = 3306$ \\ & $G_{cov} = 44112$, \texttt{v(ILP)} $ = 1411$ \\ \hline
 $\beta = 0.3$, $u=2\,SPath$ & $\beta = 0.7$, $u=2\,SPath$ \\ 
 \includegraphics[width = \factor\textwidth]{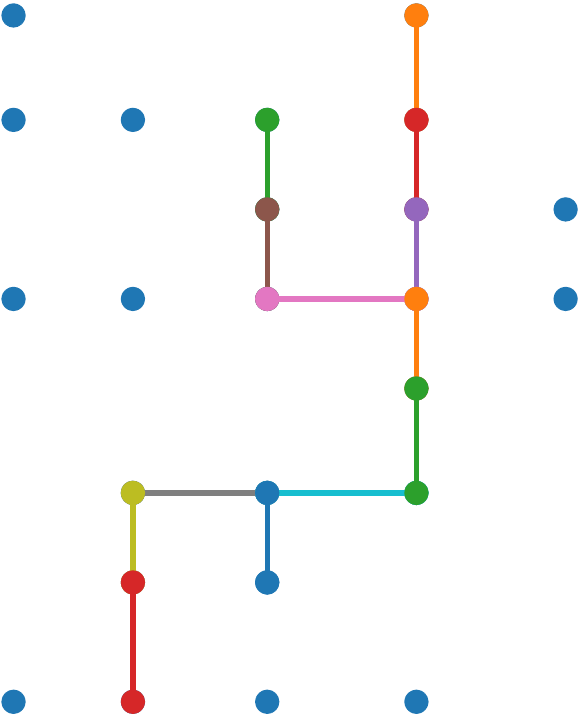} & \includegraphics[width = \factor\textwidth]{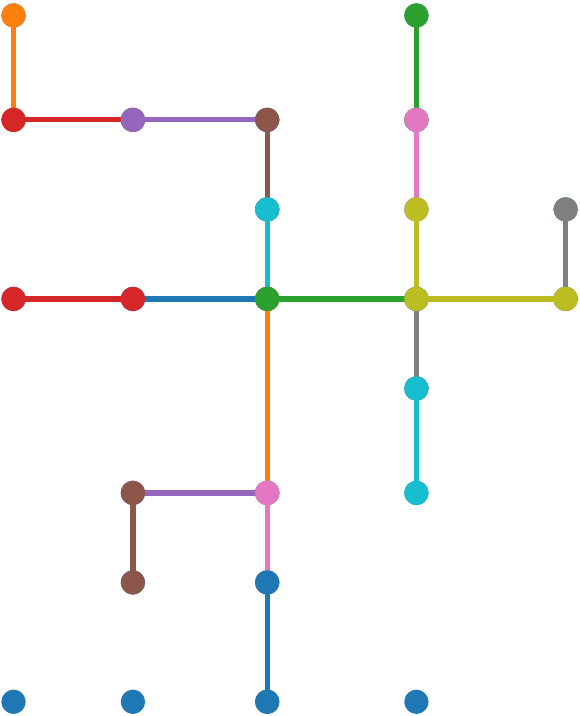} \\
 \texttt{t} $ = 925.68$, \texttt{cuts} $ = 2783$ & \texttt{t} $ = 136.06$, \texttt{cuts} $ = 3674$ \\ $G_{cov} = 24588$, \texttt{v(ILP)} $ = 1058$ & $G_{cov} = 60276$, \texttt{v(ILP)} $ = 1726$ \\ \hline
 $\beta = 0.5$, $u=1.5\,SPath$ & $\beta = 0.5$, $u=3\,SPath$ \\
 \includegraphics[width = \factor\textwidth]{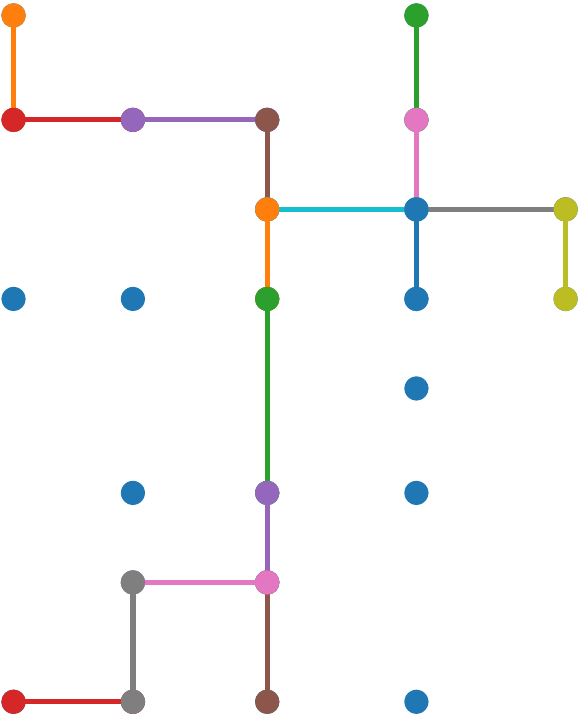} & \includegraphics[width = \factor\textwidth]{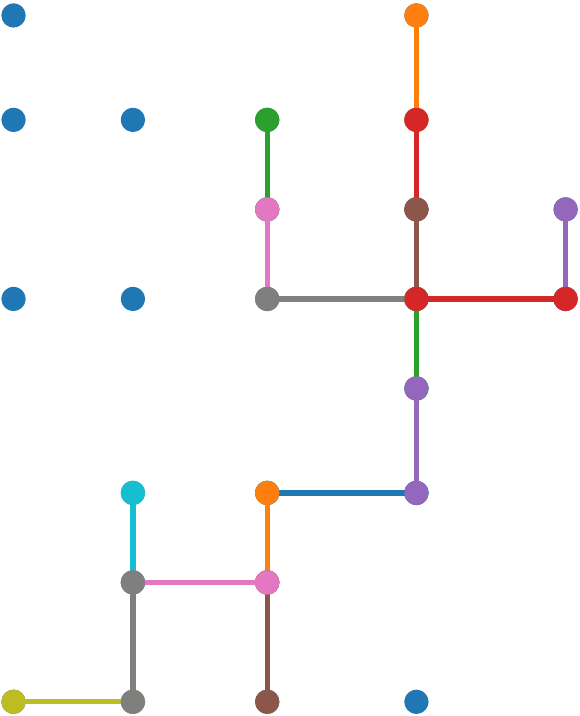} \\
 \texttt{t} $ = 1471.84$, \texttt{cuts} $ = 3793$ & \texttt{t} $ = 1149.26$, \texttt{cuts} $ = 3128$ \\ $G_{cov} = 43599$, \texttt{v(ILP)} $ = 1491$ & $G_{cov} = 42331$, \texttt{v(ILP)} $ = 1411$ \\ \hline
\end{tabular}\caption{Sensitivity analysis for the Sioux Falls Network with $(PC)$.}
\label{table: sioux PC}
\end{table}
\end{appendix}

\end{document}